\documentclass[final,leqno]{siamltex}

\usepackage{amsmath,amssymb,amscd,amsxtra,amsfonts,mathrsfs}
\usepackage{bm,epsf,graphicx,epsfig,color,latexsym,cite,cases}

\newtheorem{assumption}[theorem]{Assumption}

\title{Inverse Elastic Scattering for a Random Source}

\author{Jianliang Li\thanks{School of Mathematics and Statistics, Changsha
University of Science and Technology, Changsha, 410114, P. R. China. ({\tt
lijl@amss.ac.cn})} \and Peijun Li\thanks{Department of Mathematics, Purdue
University, West Lafayette, Indiana 47907, USA. ({\tt
lipeijun@math.purdue.edu})} 
}

\begin{document}

\maketitle

\begin{abstract}
Consider the inverse random source scattering problem for the two-dimensional
time-harmonic elastic wave equation with an inhomogeneous, anisotropic mass
density. The source is modeled as a microlocally isotropic generalized Gaussian
random function whose covariance operator is a classical pseudo-differential
operator. The goal is to recover the principle symbol of the covariance operator
from the displacement measured in a domain away from the source. For such a
distributional source, we show that the direct problem has a unique solution by
introducing an equivalent Lippmann--Schwinger integral equation. For the inverse
problem, we demonstrate that, with probability one, the principle symbol of the
covariance operator can be uniquely determined by the amplitude of
the displacement averaged over the frequency band, generated by a single
realization of the random source. The analysis employs the Born approximation,
asymptotic expansions of the Green tensor, and microlocal analysis of the
Fourier integral operators.
\end{abstract}



\pagestyle{myheadings}
\thispagestyle{plain}
\markboth{J. Li and P. Li}{Inverse Elastic Scattering for a Random Source}

\section{Introduction}

The inverse source scattering problems are to recover the unknown sources from
the radiated wave field which is generated by the unknown sources. These
problems are motivated by significant applications in diverse scientific areas
such as medical imaging \cite{ABF, FKM, NOH}, and antenna design and synthesis
\cite{DML}. Driven by these applications, the inverse source scattering problems
have been extensively studied by many researchers in both mathematical and
engineering communities. Consequently, a great deal of mathematical and
numerical results are available, especially for deterministic sources
\cite{ACTV, BN, BLRX, DML, EV}. It is known that the inverse source problem,
in general, does not have a unique solution at a single frequency due to the
existence of non-radiating sources \cite{BC, BC-JMP, DS, HKP}. There are two
approaches to overcome the issue non-uniqueness: one is to seek the minimum
energy solution \cite{MD}, which represents the pseudo-inverse solution for the
inverse source problem; the other is the use of multi-frequency data to achieve
uniqueness and gain increasing stability \cite{BLLT, BLT, BLZ, CIL, LY}. 

In many situations, the source, hence the wave field, may not be deterministic
but are rather modeled by random processes \cite{BAZC}. Due to the extra
challenge of randomness and uncertainties, little is known for the inverse
random source scattering problems. In \cite{BCL16, BCLZ, LCL, PL, BCL18, BX},
the random source was assumed to be driven by an additive white noise.
Mathematical modeling and numerical computation were proposed for a class of
inverse source problems for acoustic and elastic waves. The method
requires to know the expectation of the scattering data, which needs to be
measured corresponding to a fairly large number of realizations of the source. 

Recently, a different model is proposed in \cite{CHL, LPS} to describe random
functions. The random function is considered to be a generalized Gaussian random
function whose covariance is represented by a classical pseudo-differential
operator. The authors studied an inverse problem for the
two-dimensional random Schr\"{o}dinger equation where the potential function was
random. It is shown that the principle symbol of the covariance operator can be
uniquely determined by the backscattered far field \cite{CHL} or backscattered
field \cite{LPS}, generated by a single realization of the random potential and
plane waves \cite{CHL} or a point source \cite{LPS} as the incident field. A
related work can be found in \cite{HLP} where the authors considered an inverse
scattering problem in a half-space with an impedance boundary condition where
the impedance function was random. In \cite{LHL}, the inverse random source
scattering problems were considered for the time-harmonic acoustic and elastic
waves in a homogeneous and isotropic medium. The source is assumed to be a
microlocally isotropic generalized Gaussian random function. It is shown that
the amplitude of the scattering field averaged over the frequency band, obtained
from a single realization of the random source, determines uniquely the
principle symbol of the covariance operator. In this paper, we study an inverse
random source scattering problem for the two-dimensional elastic wave equation
with an inhomogeneous, anisotropic mass density. This paper significantly
extends our previous work on the inverse random source problem for elastic
waves. The techniques also differ greatly because a more complicated model
equation is considered. 

The wave propagation is governed by the stochastic
elastic wave equation
\begin{eqnarray}\label{a1}
\mu\Delta \bm{u}+(\lambda+\mu)\nabla\nabla\cdot
\bm{u}+\omega^2\bm{u}-\bm{M}\bm{u}=\bm{f} \quad {\rm in} ~ \mathbb R^{2},
\end{eqnarray}
where $\bm u\in\mathbb C^2$ is the complex-valued displacement vector,
$\omega>0$ is the angular frequency, $\lambda$ and $\mu$ are the Lam\'{e}
constants satisfying $\mu>0, \lambda+\mu>0$, and $\bm{M}\in\mathbb R^{2\times
2}$ is a deterministic real-valued symmetric matrix with a compact support
contained in $D\subset\mathbb R^2$ and represents either a linear load acting
on the elastic medium or an inhomogeneous, anisotropic mass density of the
elastic medium inside $D$. The randomness of \eqref{a1} comes from the external
source $\bm{f}=(f_1,f_2)^\top$. Throughout, we make the following assumption. 

\begin{assumption}\label{assu1}
The domain $D$ is bounded, simply connected, and Lipschitz. The source
$\bm{f}=(f_1,f_2)^\top$ is compactly supported in $D$ and $f_j, j =1, 2$ are
microlocally isotropic Gaussian random fields of the same order $m\in [2,
\frac{5}{2})$ in $D$. Each covariance operator $C_{f_j}$ is a classical
pseudo-differential operator having the same principle symbol
$\phi(x)|\xi|^{-m}$ with $\phi\in C_0^{\infty}(D), \phi\geq 0$.
Moreover, the source $\bm{f}$ is assumed to be bounded almost surely
with $\mathbb E(f_j)=0$ and $\mathbb E(f_1 f_2)=0$.
\end{assumption}

Since \eqref{a1} is imposed in the whole space $\mathbb R^{2}$, an
appropriate radiation condition is needed to complete the problem formulation.
By the Helmholtz decomposition, the displacement $\bm{u}$ can be decomposed into
the compressional part $\bm{u}_{\rm p}$ and the shear part $\bm{u}_{\rm s}$ away
from the source:
\[
\bm{u}=-\frac{1}{\kappa^2_{\rm p}}\nabla\nabla\cdot\bm{u}+\frac{1}{\kappa^2_{\rm
s}}{\bf curl}{\rm curl}\bm{u}:=\bm{u}_{\rm p}+\bm{u}_{\rm s}
\quad{\rm in} ~ \mathbb R^2\setminus \overline{D}.
\]
For a scalar function $u$ and a vector function ${\bm u}=(u_1, u_2)^\top$, the
vector and scalar cur operators are defined by  
\[
{\bf curl}u=(\partial_{x_2}u,-\partial_{x_1}u)^\top,\quad {\rm
curl}{\bm u}=\partial_{x_1} u_2-\partial_{x_2}u_1.
\]
The Kupradze--Sommerfeld radiation condition requires that $\bm{u}_{\rm p}$ and
$\bm{u}_{\rm s}$ satisfy the Sommerfeld radiation condition:
\begin{eqnarray}\label{a2}
\lim_{r\rightarrow\infty}r^{\frac{1}{2}}\left(\partial_r
\bm{u}_{\rm p} -{\rm i}\kappa_{\rm p}\bm{u}_{\rm p}\right)=0,\quad
\lim_{r\rightarrow\infty}r^{\frac{1}{2}}\left(\partial_r
\bm{u}_{\rm s} -{\rm i}\kappa_{\rm s}\bm{u}_{\rm s}\right)=0,\quad r=|x|,
\end{eqnarray}
where $\kappa_{\rm p}$ and $\kappa_{\rm s}$ are known as the compressional
wavenumber and the shear wavenumber, respectively, and are defined by 
\begin{eqnarray*}
\kappa_{\rm p}=\frac{\omega}{(\lambda+2\mu)^{1/2}}=c_{\rm p}\omega,\qquad
\kappa_{\rm s}=\frac{\omega}{\mu^{1/2}}=c_{\rm s}\omega.
\end{eqnarray*}
Here 
\begin{eqnarray*}
c_{\rm p} = (\lambda+2\mu)^{-1/2},\quad c_{\rm s} = \mu^{-1/2}.
\end{eqnarray*}
Note that $c_{\rm p}$ and $c_{\rm s}$ are independent of $\omega$ and
$c_{\rm p}<c_{\rm s}$.

Given $\omega, \lambda, \mu, \bm{M}$, and $\bm{f}$, {\em the direct scattering
problem} is to determine $\bm{u}$ which satisfies \eqref{a1}--\eqref{a2}. For
$m\in [2,5/2)$, the random source is a rough field and belongs to the Sobolev
space with a negative smoothness index almost surely. A careful study is needed
to show the well-posedness of the direct scattering problem for such a
distributional source. Using Green's theorem and the Kupradze--Sommerfeld
radiation condition, we show that the direct scattering problem is equivalent to
a Lippmann--Schwinger equation. By the Fredholm alternative along with the
unique continuation principle, we prove that the Lippmann--Schwinger equation
has a unique solution which belongs to the Sobolev space with a negative
smoothness index almost surely. Thus the well-posedness is established for the
direct scattering problem. 

Given $\omega, \lambda, \mu, \bm{M}$, {\em the inverse scattering problem} is to
determine  $\phi(x)$, the micro-correlation strength of the source, from the
displacement measured in a bounded domain $U\subset\mathbb R^2\setminus\overline
D$ standing for the measurement domain, which is required to satisfy the
following assumption. 

\begin{assumption}\label{assu2}
The measurement domain $U$ is bounded, simply connected, Lipschitz, convex, 
and has a positive distance to $D$. 
\end{assumption}

In addition, the following assumption is imposed on $\bm{M}$.

\begin{assumption}\label{assu3}
The matrix $\bm{M}=(M_{ij})_{2\times 2}$ is a deterministic and
real-valued symmetric matrix with $M_{ij}\in C_0^1(\overline{D})$ for $i,
j=1,2$.
\end{assumption}

The following result concerns the uniqueness of the inverse scattering problem
and is the main result of this paper. 

\begin{theorem}\label{theorem1}
Let $\bm{f}, U$, and $\bm M$ satisfy Assumptions \ref{assu1}, \ref{assu2}, and
\ref{assu3}, respectively. Then for all $x\in U$, it
holds almost surely that 
\begin{eqnarray}\label{a3}
\lim_{Q\rightarrow\infty}\frac{1}{Q-1}\int_1^Q\omega^{m+1}|\bm{u}(x,
\omega)|^2 d\omega=a\int_{\mathbb R^2}\frac{1}{|x-y|}\phi(y)dy,
\end{eqnarray}
where $a=\frac{1}{32\pi}\left({c_{\rm s}^{3-m}}+{c_{\rm
p}^{3-m}}\right)$ is a constant. Moreover, the function $\phi$ can be uniquely
determined from the integral equation \eqref{a3} for all $x\in U$.
\end{theorem} 

For any finite $Q$, the scattering data given in the left-hand side of
(\ref{a3}) is random in the sense of that it depends on the realization of the
source, while (\ref{a3}) shows that in the limit $Q\to \infty$, the scattering
data becomes statistically stable, i.e., it is independent of realization of the
source. Hence, Theorem \ref{theorem1} shows that the amplitude of the
displacement averaged over the frequency band, measured from a single
realization of the random source, can uniquely determine the micro-correlation
strength function $\phi$. The proof of Theorem \ref{theorem1} combines the
Born approximation, asymptotic expansions of the Green tensor, 
and microlocal analysis of integral operators

The paper is organized as follows. In Section 2, we briefly introduce some
necessary notations including the Sobolev spaces, the generalized Gaussian
random function, and some properties of the Hankel function of the first kind.
Section 3 addresses the direct scattering problem; Sections 4 and 5 study the
inverse scattering problem. In Section 3, the well-posedness of the direct
scattering problem is established for a distributional source. Using the
Riesz--Fredholm theory and the Sobolev embedding theorem, we show that the
direct scattering problem is equivalent to a uniquely solvable
Lippmann--Schwinger equation. Section 4 presents the Born approximation of the
solution to the Lippmann--Schwinger integral equation. Section 5
examines the second term in the Born approximation via the microlocal analysis. 
The paper is concluded with some general remarks in Section 6.

\section{Preliminaries}

In this section, we introduce some notations and properties of the Sobolev
spaces, the generalized Gaussian random functions, and the Hankel function of
the first kind. 

\subsection{Sobolev spaces}

Let $C_0^{\infty}(\mathbb R^2)$ be the set of smooth functions with compact
support, and $\mathcal{D'}(\mathbb R^2)$ be the set of generalized
(distributional) functions. Given $1<p<\infty, s\in\mathbb R$, define the
Sobolev space 
\[
H^{s,p}(\mathbb R^2)=\{h=(I-\Delta)^{-\frac{s}{2}}g : g\in L^p(\mathbb R^2)\},
\]
which has the norm
\[
\|h\|_{H^{s,p}(\mathbb R^2)}=\|(I-\Delta)^{\frac{s}{2}}h\|_{L^p(\mathbb R^2)}.
\]
With the definition of Sobolev spaces in the whole space, 
the Sobolev space $H^{s,p}(V)$ for any Lipschitz domain $V\subset \mathbb R^2$
can be defined as the restriction to $V$ of the elements in $H^{s,p}(\mathbb
R^2)$ with the norm
\[
\|h\|_{H^{s,p}(V)}=\inf\{\|g\|_{H^{s,p}(\mathbb R^2)} : g|_V=h\}.
\]
By \cite{JK}, for $s\in\mathbb R$ and $1<p<\infty$,
$H^{s,p}_0(V)$ can be defined as the space of all distributions $h\in
H^{s,p}(\mathbb R^2)$
satisfying ${\rm supp} h\subset\overline{V}$ with the norm 
\[
\|h\|_{H^{s,p}_0(V)}=\|h\|_{H^{s,p}(\mathbb R^2)}.
\]
It is known that $C_0^{\infty}(V)$ is dense in $H_0^{s,p}(V)$ for any
$1<p<\infty, s\in\mathbb R$; $C_0^{\infty}(V)$ is dense in $H^{s,p}(V)$ for any
$1<p<\infty, s\leq 0$; $C^{\infty}(\overline{V})$ is dense in $H^{s,p}(V)$ for
any  $1<p<\infty, s\in\mathbb R$. In addition, by \cite[Propositions 2.4 and
2.9]{JK}, for any $s\in\mathbb R$ and $p, q\in (1,\infty)$ satisfying
$\frac{1}{p}+\frac{1}{q}=1$, we have
\[
H^{-s,q}_0(V)=(H^{s,p}(V))'\quad {\rm and}\quad H^{-s,q}(V)=(H_0^{s,p}(V))',
\]
where the prime denotes the dual space. 

The following two lemmas will be used in the subsequent analysis. The proofs of
Lemma \ref{lemma1} and Lemma \ref{lemma2} can be found in \cite[Lemma 2]{LPS}
and \cite[Proposition 1]{MT}, respectively. 

\begin{lemma}\label{lemma1}
Assume that $\epsilon>0$, $1<r<\infty$, $\frac{1}{r}+\frac{1}{r'}=1$, $g\in
H^{\epsilon,2r}_{\rm loc}(\mathbb R^2)$, $h\in H_0^{-\epsilon,r'}(\mathbb R^2)$.
Then $gh\in H_0^{-\epsilon,\tilde{r}}(\mathbb R^2)$ and satisfies
\begin{eqnarray*}
\|gh\|_{H_0^{-\epsilon,\tilde{r}}(\mathbb R^2)}\lesssim
\|g\|_{H^{\epsilon,2r}(\mathbb R^2)}||h||_{ H_0^{-\epsilon,r'}(\mathbb R^2)},
\end{eqnarray*}
where $\tilde{r}=\frac{2r}{2r-1}$. 
\end{lemma}

\begin{lemma}\label{lemma2}
Assume that $s>0$, $1<\tilde{p}<\infty$ and
$\frac{1}{\tilde{p}}=\frac{1}{q_1}+\frac{1}{q_2}=\frac{1}{r_1}+\frac{1}{r_2},
q_1 , r_1\in (1,\infty], q_2, r_2\in (1,\infty)$. Then the following estimate
holds
\begin{eqnarray*}
\|gh\|_{H^{s,\tilde{p}}(\mathbb
R^2)}\lesssim\|g\|_{L^{q_1}(\mathbb R^2)}||h||_{H^{s,q_2}(\mathbb R^2)}+
\|h\|_{L^{r_1}(\mathbb R^2)}\|g\|_{H^{s,r_2}(\mathbb R^2)}. 
\end{eqnarray*}
\end{lemma}

Throughout the paper, $a\lesssim b$ stands for $a\leq C b$, where $C$ is a
positive constant and its specific value is not required but should be clear
from the context.

\subsection{Generalized Gaussian random functions}

Let $(\Omega, \mathcal{F},\mathcal{P})$ be a complete probability space. The
function $h$ is said to be a generalized Gaussian random function if $h:
\Omega\rightarrow\mathcal{D}'(\mathbb R^2)$ is a mapping such that, for
each $\hat{\omega}\in\Omega$, the realization $h(\hat{\omega})$ is a real-valued
linear functional on $C_0^{\infty}(\mathbb R^2)$ and the function 
\begin{eqnarray*}
\hat{\omega}\in\Omega\rightarrow \langle h(\hat{\omega}),\psi\rangle\in\mathbb R
\end{eqnarray*}
is a Gaussian random variable for all $\psi\in C_0^{\infty}(\mathbb R^2)$. The
distribution of $h$ is determined by its expectation $\mathbb Eh$ and the
covariance ${\rm Cov}h$ defined as follows
\begin{eqnarray*}
\mathbb Eh: \psi\in C_0^{\infty}(\mathbb R^d)&\longmapsto& \mathbb E\langle
h,\psi\rangle\in\mathbb R,\\
{\rm Cov} h:(\psi_1,\psi_2)\in C_0^{\infty}(\mathbb R^d)^2&\longmapsto& {\rm
Cov}(\langle h,\psi_1\rangle,\langle h,\psi_2\rangle)\in\mathbb R,
\end{eqnarray*}
where $\mathbb E\langle h,\psi\rangle$ denotes the expectation of $\langle
h,\psi\rangle$ and 
\begin{eqnarray*}
{\rm Cov}(\langle h,\psi_1\rangle,\langle h,\psi_2\rangle)=\mathbb E((\langle
h,\psi_1\rangle-\mathbb E\langle h,\psi_1\rangle)(\langle
h,\psi_2\rangle-\mathbb E\langle h,\psi_2\rangle))
\end{eqnarray*}
denotes the covariance of $\langle h, \psi_1\rangle$ and $\langle
h,\psi_2\rangle$. The covariance operator ${\rm Cov}_{h}: C_0^{\infty}(\mathbb
R^2)\rightarrow \mathcal{D'}(\mathbb R^2)$ is defined by 
\begin{eqnarray}\label{b1}
\langle {\rm Cov}_{h}\psi_1,\psi_2\rangle={\rm Cov}(\langle
h,\psi_1\rangle,\langle h,\psi_2\rangle)=\mathbb E(\langle h-\mathbb
Eh,\psi_1\rangle \langle h-\mathbb Eh,\psi_2\rangle).
\end{eqnarray}
Since the covariance operator ${\rm Cov}_h$ is continuous, the Schwartz kernel
theorem shows that there exists a unique $C_h\in\mathcal{D}'(\mathbb R^2\times
\mathbb R^2)$, usually called the covariance function, such that
\begin{eqnarray}\label{b2}
\langle C_h,\psi_1\otimes\psi_2\rangle=\langle  {\rm
Cov}_{h}\psi_1,\psi_2\rangle,\quad\forall \psi_1,\psi_2\in C_0^{\infty}(\mathbb
R^2).
\end{eqnarray} 
By (\ref{b1}) and (\ref{b2}), it is easy to see that 
\begin{eqnarray*}
C_h(x,y)=\mathbb{E}((h(x)-\mathbb{E}h(x))(h(y)-\mathbb{E}h(y))).
\end{eqnarray*}

In this paper, we are interested in the generalized, microlocally isotropic
Gaussian random function which is defined as follows. 

\begin{definition}
A generalized Gaussian random function $h$ on $\mathbb R^2$ is called
microlocally isotropic of order $m$ in D, if the realizations of $h$ are almost
surely supported in the domain $D$ and its covariance operator ${\rm Cov}_h$ is
a classical pseudo-differential operator having the principal symbol
$\phi(x)|\xi|^{-m}$, where $\phi\in C_0^{\infty}(\mathbb R^2)$ satisfies ${\rm
supp} \phi\subset D$ and $\phi(x)\geq 0$ for all $x\in \mathbb R^2$.
\end{definition}

In particular, we pay attention to the case $m\in[2,5/2)$, which corresponds to
rough fields. The following results will also be used in the subsequent
analysis. The proofs of Lemmas \ref{lemma3} and \ref{lemma4} can be found
in \cite[Theorem 2 and Proposition 1]{LPS}. 

\begin{lemma}\label{lemma3}
Let $h$ be a generalized, microlocally isotropic Gaussian random function of
order $m$ in $D$. If $m=2$, then $h \in H^{-\varepsilon, p}(D)$ almost
surely for all $\varepsilon>0, 1<p<\infty$.  If $m\in(2,5/2)$,
then $h\in C^{\alpha}(D)$ almost surely for all $\alpha\in(0,
\frac{m-2}{2})$.
\end{lemma}

\begin{lemma}\label{lemma4}
Let $h$ be a microlocally isotropic Gaussian random field of order
$m\in[2,5/2)$. Then the Schwartz kernel of the covariance operator ${\rm Cov}_h$
has the form
\begin{eqnarray*}
C_h(x, y)=\begin{cases}
           c_0(x,y){\rm log}|x-y|+r_1(x,y)&\quad{\rm for}\quad m=2,\\
           c_0(x,y)|x-y|^{m-2}+r_1(x,y)&\quad{\rm for}\quad m\in(2,5/2),
          \end{cases}
\end{eqnarray*}
where $c_0\in C_0^{\infty}(D\times D)$ and $r_1\in C_0^{\alpha}(D\times D)$ for
any $\alpha<1$.
\end{lemma}

\subsection {The properties of the Hankel function of the first kind}

In this subsection, we present some asymptotic expansions of the Hankel function
of the first kind for small and large arguments. Let $H_n^{(1)}$ be the
Hankel function of the first kind. Recall the definition
\begin{eqnarray*}\label{b5}
H_n^{(1)}(t)=J_n(t)+{\rm i}Y_n(t),
\end{eqnarray*}
where $J_n$ and $Y_n$ are the Bessel functions of the first and second kind
with order $n$, respectively. They admit the following expansions
\begin{align}\label{b6}
J_n(t)&=\sum_{p=0}^{\infty}\frac{(-1)^p}{p!(n+p)!}\left(\frac{t}{2}\right)^{
n+2p},\\\notag
Y_n(t)&=\frac{2}{\pi}\left\{\ln\frac{t}{2}+\gamma\right\}J_n(t)-\frac{1}{\pi}
\sum_{p=0}^{n-1}\frac{(n-1-p)!}{p!}\left(\frac{2}{t}\right)^{n-2p}\\\label{b7}
&\qquad-\frac{1}{\pi}\sum_{p=0}^{\infty}\frac{(-1)^p}{p!(n+p)!}\left(\frac{t}{2}
\right)^{n+2p}\{\psi(p+n)+\psi(p)\},
\end{align}
where
$\gamma:=\lim_{p\rightarrow\infty}\left\{\sum_{j=1}^p j^{-1}
-\ln p\right\}$ denotes the Euler constant,  $\psi(0)=0$,
$\psi(p)=\sum_{j=1}^p j^{-1}$, and the finite sum in
(\ref{b7}) is set to be zero for $n=0$. 

Using the expansions (\ref{b6}) and (\ref{b7}), we may verify as $t\to 0$ that
\begin{align}\label{b8}
H_0^{(1)}(t)&=\frac{2{\rm
i}}{\pi}\ln\frac{t}{2}+b_0+O(t^2\ln\frac{t}{2}),\\\label{b9}
H_1^{(1)}(t)&=-\frac{2{\rm i}}{\pi}\frac{1}{t}+\frac{{\rm
i}}{\pi}t\ln\frac{t}{2}+b_1t+O(t^3\ln\frac{t}{2}),\\\label{b10}
H_2^{(1)}(t)&=-\frac{4{\rm i}}{\pi}\frac{1}{t^2}-\frac{\rm i}{\pi}+\frac{\rm
i}{4\pi}t^2\ln\frac{t}{2}+b_2t^2+O(t^4\ln\frac{t}{2}),\\\label{b11}
H_3^{(1)}(t)&=-\frac{16{\rm i}}{\pi}\frac{1}{t^3}-\frac{2{\rm
i}}{\pi}\frac{1}{t}-\frac{\rm i}{4\pi}t+\frac{\rm
i}{24\pi}t^3\ln\frac{t}{2}+b_3t^3+O(t^5\ln\frac{t}{2}),
\end{align}
where $b_0=1+\frac{2{\rm i}}{\pi}\gamma$, $b_1=\frac{1}{2}+\frac{{\rm
i}}{\pi}\gamma-\frac{\rm i}{2\pi}$, $b_2=\frac{\gamma {\rm i}}{4\pi}-\frac{3{\rm
i}}{16\pi}+\frac{1}{8}, b_3=\frac{{\gamma\rm
i}}{24\pi}+\frac{1}{48}-\frac{11{\rm i}}{288\pi}$. Denote
\[
\Gamma_n(z,\omega):=\kappa^n_{\rm s}H_n^{(1)}(\kappa_{\rm s}|z|)-\kappa^n_{\rm
p}H_n^{(1)}(\kappa_{\rm p}|z|).
\]
Noting (\ref{b8})--(\ref{b11}), we have from a direct calculation as $|z|\to 0$
that
\begin{align}\label{b12}
&\Gamma_1(z,\omega)=\frac{\rm
i}{\pi}|z|\left(\kappa_{\rm s}^2\ln\frac{\kappa_{\rm s}|z|}{2}-\kappa_{\rm
p}^2\ln\frac{\kappa_{\rm p}|z|}{2}\right)
+b_1(\kappa_{\rm s}^2-\kappa_{\rm
p}^2)|z|+O(|z|^3\ln\frac{|z|}{2}),\\\label{b13}
&\Gamma_2(z,\omega)=\frac{\rm
i}{4\pi}\left(\kappa_{\rm s}^4\ln\frac{\kappa_{\rm s}|z|}{2} -\kappa_{\rm
p}^4\ln\frac{\kappa_{\rm p}|z|}{2}\right)|z|^2
-\frac{\rm i}{\pi}(\kappa_{\rm s}^2-\kappa_{\rm
p}^2)+O(|z|^2),\\\label{b14}
&\Gamma_3(z,\omega)=\frac{2{\rm i}}{\pi}(\kappa_{\rm p}^2-\kappa_{\rm
s}^2)\frac{1}{|z|}+\frac{\rm i}{4\pi}(\kappa_{\rm p}^4-\kappa_{\rm
s}^4)|z|+O(|z|^3\ln\frac{|z|}{2}).
\end{align}

For a large argument, i.e., as $|z|\to \infty$, it follows from
\cite[(9.2.7)--(9.2.10)]{AS} and \cite[(5.11.4)]{NN} that the Hankel function
of the first kind  $H_n^{(1)}$ has the following asymptotics 
\begin{align}\label{b15}
H_n^{(1)}(z)=&\sqrt{\frac{1}{ z}}e^{{\rm i}(
z-(\frac{n}{2}+\frac{1}{4})\pi)}\notag\\
&\times\Big(\sum_{j=0}^{N}a_j^{(n)}z^{-j}+O(|z|^{-N-1})\Big),\quad |{\rm
arg}z|\leq \pi-\delta,
\end{align}
where $\delta$ is a small positive number and the
coefficients $a_j^{(n)}=(-2{\rm i})^j\sqrt{\frac{2}{\pi}}(n,j)$ with
\[
(n,0)=1, \quad (n,j)=\frac{(4n^2-1)(4n^2-3^2)\cdots(4n^2-(2j-1)^2)}{2^{2j}j!}. 
\]
Using the first $N$ terms in the asymptotic of $H_n^{(1)}(\kappa |z|)$, we
define 
\begin{eqnarray}\label{b16}
H_{n,N}^{(1)}(\kappa |z|):=\sqrt{\frac{1}{\kappa |z|}}e^{{\rm i}(\kappa
|z|-(\frac{n}{2}+\frac{1}{4})\pi)}\sum_{j=0}^N a_j^{(n)}\left(\frac{1}{\kappa
|z|}\right)^j.
\end{eqnarray}
Denote $\Gamma_{n,N}(\kappa |z|):=H_n^{(1)}(\kappa |z|)-H_{n,N}^{(1)}(\kappa
|z|)$, it is easy to show from (\ref{b15}) that 
\begin{eqnarray}\label{b17}
\big|\Gamma_{n,N}(\kappa |z|)\big|\leq c\bigg(\frac{1}{\kappa
|z|}\bigg)^{N+\frac{3}{2}}.
\end{eqnarray}

\section{The direct scattering problem}

This section aims to establish the well-posedness of the direct scattering
problem for a distributional source. Based on Green's theorem and the
Kupradze--Sommerfeld radiation, the direct problem is equivalently formulated as
a Lippmann--Schwinger equation, which is shown to have a unique solution by
using the Riesz--Fredholm theory and the Sobolev embedding theorem. 

By Lemma \ref{lemma3}, we have that $\bm{f}\in H^{-\varepsilon,
p}(D)^2$ almost surely for all $\varepsilon>0$, $1<p<\infty$ if $m=2$;
$\bm{f}\in C^{0,\alpha}(D)^2$ almost surely for all $\alpha\in
(0,\frac{m-2}{2})$ if $m\in (2,5/2)$. Therefore, it suffices to show that the
scattering problem (\ref{a1})--(\ref{a2}) has a unique solution for such
a deterministic source $\bm{f}\in H^{-\varepsilon, p}(D)^2$. 

Introduce the Green tensor $\bm{G}(x,y,\omega)\in\mathbb C^{2\times 2}$ for the
Navier equation
\begin{eqnarray}\label{c1}
\bm{G}(x,y,\omega)=\frac{1}{\mu}\Phi(x,y,\kappa_{\rm
s})\bm{I}+\frac{1}{\omega^2}\nabla_x\nabla_x^\top(\Phi(x,y,
\kappa_{\rm s})-\Phi(x,y,\kappa_{\rm p})),
\end{eqnarray}
where $\bm{I}$ is the $2\times 2$ identity matrix, $\Phi(x,y,\kappa)=\frac{\rm
i}{4}H_0^{(1)}(\kappa|x-y|)$ is the fundamental solution for the two-dimensional
Helmholtz equation, and $\nabla_x\nabla_x^\top$ is defined by 
\begin{eqnarray*}
\nabla_x\nabla_x^\top\varphi=\begin{bmatrix} \partial^2_{x_1x_1}\varphi &
\partial^2_{x_1x_2}\varphi\\[3pt] \partial^2_{x_2x_1}\varphi &
\partial^2_{x_2x_2}\varphi
\end{bmatrix}
\end{eqnarray*}
for some scalar function $\varphi$ defined in $\mathbb R^2$. It is easy to
note that the Green tensor $\bm{G}(x,y,\omega)$ is symmetric with
respect to the variables $x$ and $y$.

In order to obtain the well-posedness of the scattering problem
(\ref{a1})--(\ref{a2}), we first derive a Lippmann--Schwinger equation which is
equivalent to the direct scattering problem, then we show that the
Lippmann--Schwinger equation has a unique solution. 

\begin{theorem}\label{theorem2}
For some $p\geq 2$, $\frac{1}{p}+\frac{1}{p'}=1$, $0<\varepsilon<\frac{2}{p}$,
$\bm{f}\in H_0^{-\varepsilon, p'}(D)^2$, if $\bm{M}$ satisfies
Assumption \ref{assu3}, then the scattering problem (\ref{a1})--(\ref{a2}) is
equivalent to the Lippmann--Schwinger equation
\begin{eqnarray}\label{c2}
\bm{u}(x)+\int_D\bm{G}(x,y,\omega)\bm{M}(y)\bm{u}(y)dy=-\int_D\bm{G}(x,y,
\omega)\bm{f}(y)dy,\quad x\in\mathbb R^2.
\end{eqnarray}
\end{theorem}

\begin{proof}
Let $\bm{u}\in H^{\varepsilon,p}_{\rm loc}(\mathbb R^2)^2$ be a solution to
(\ref{c2}), then we have 
\begin{eqnarray*}
\bm{u}(x)=-\int_D\bm{G}(x,y,\omega)\bm{M}(y)\bm{u}(y)dy-\int_D\bm{G}(x,y,
\omega)\bm{f}(y)dy,\quad x\in\mathbb R^2.
\end{eqnarray*}
Since the Green tensor $\bm{G}(x,y,\omega)$ and its derivatives satisfy the
Kupradze--Sommerfeld radiation condition, we conclude that $\bm{u}$ also
satisfies the Kupradze--Sommerfeld radiation condition. By \eqref{c1}, the
Green tensor $\bm{G}(x,y,\omega)$ satisfies 
\begin{eqnarray}\label{c3}
\mu\Delta\bm{G}(x,y,\omega)+(\lambda+\mu)\nabla\nabla\cdot\bm{G}(x,y,
\omega)+\omega^2\bm{G}(x,y,\omega)=-\delta(x-y)\bm{I}.
\end{eqnarray}
Letting $y=0$ and taking the Fourier transform with respect to $x$ on both side
of (\ref{c3}) yields 
\begin{eqnarray}\label{c4}
\bm{\widehat{G}}(\xi)=\left[(4\pi^2\mu|\xi|^2-\omega^2)\bm{I}
+4\pi^2(\lambda+\mu)\xi\cdot\xi^\top\right]^{-1},\quad\xi\in\mathbb R^3.
\end{eqnarray}
Note that the integral in (\ref{c2}) is a convolution since $\bm{G}(x,y,\omega)$
is a function of $x-y$. Taking the Fourier transform on both sides of
(\ref{c2}) and using (\ref{c4}) lead to
\begin{eqnarray*}
\bm{\hat{u}}(\xi)=-\left[(4\pi^2\mu|\xi|^2-\omega^2)\bm{I}
+4\pi^2(\lambda+\mu)\xi\cdot\xi^\top\right]^{-1}(\bm{\hat{f}}(\xi)+\widehat{\bm{
M}\bm{u}}(\xi)),
\end{eqnarray*}
which gives 
\begin{eqnarray*}
\left[(4\pi^2\mu|\xi|^2-\omega^2)\bm{I}
+4\pi^2(\lambda+\mu)\xi\cdot\xi^\top\right
]\hat{\bm{u}}(\xi)+\widehat{\bm{M}\bm{u}}(\xi)=-\hat{\bm{f}}(\xi).
\end{eqnarray*}
Taking the inverse Fourier transform yields
\begin{eqnarray*}
\mu\Delta \bm{u}+(\lambda+\mu)\nabla\nabla\cdot
\bm{u}+\omega^2\bm{u}-\bm{M}\bm{u}=\bm{f} \quad {\rm in} ~ \mathbb R^{2}.
\end{eqnarray*}
Hence, $\bm{u}$ is the solution of the direct scattering problem
(\ref{a1})--(\ref{a2}).

Conversely, if $\bm{u}$ is a solution of the direct scattering problem
(\ref{a1})--(\ref{a2}), we show that $\bm{u}$ satisfies the Lippmann--Schwinger
equation (\ref{c2}). Since 
\begin{eqnarray*}\label{c5}
\mu\Delta \bm{u}+(\lambda+\mu)\nabla\nabla\cdot
\bm{u}+\omega^2\bm{u}=\bm{M}\bm{u}+\bm{f} \quad {\rm in} ~ \mathbb R^{2},
\end{eqnarray*}
where $\bm{u}\in H_{\rm loc}^{\varepsilon,p}(\mathbb R^2)^2$ and $M_{ij}\in
C_0^1(\overline{D})$. Note that $\bm{f}\in H_0^{-\varepsilon, p'}(\mathbb
R^2)^2$, we have that  $\bm{M}\bm{u}+\bm{f}\in H_0^{-\varepsilon, p'}(\mathbb
R^2)^2$. An application of Lemma 4.1 in \cite{LHL} shows that for some fixed
$x\in \mathbb R^2$, $\bm{G}(x,\cdot,\omega)\in [L^2_{\rm loc}(\mathbb R^2)\cap
H_{\rm loc}^{1,\hat{p}}(\mathbb R^2)]^{2\times 2}$ for $\hat{p}\in (1,2)$. Since
$0<\varepsilon<\frac{2}{p}$, a simple calculation gives that
$\frac{1}{p}-\frac{\varepsilon}{2}>0$. Let
$\widetilde{\delta}=\frac{1}{p}-\frac{\varepsilon}{2}$ and define
$\widetilde{p}:=\frac{2}{1+\widetilde{\delta}}<2$, then
$\frac{1}{\widetilde{p}}-\frac{1}{2}<\frac{1}{p}-\frac{\varepsilon}{2}$. It
follows from the Sobolev embedding theorem that $H_{\rm
loc}^{1,\widetilde{p}}(\mathbb R^2)$ is embedded into $H_{\rm loc}^{\varepsilon,
p}(\mathbb R^2)$, which implies that $\bm{G}(x,\cdot,\omega)\in [H_{\rm
loc}^{\varepsilon, p}(\mathbb R^2)]^{2\times
2}$. Choose a large enough ball $B_r$ such that $D\subset B_r$, then we have in
the sense of distributions that  
\begin{align*}
\int_{B_r}\bm{G}(x,y,\omega)\left[\mu\Delta
\bm{u}(y)+(\lambda+\mu)\nabla\nabla\cdot
\bm{u}(y)+\omega^2\bm{u}(y)\right]dy\\
=\int_{B_r}\bm{G}(x,y,\omega)[\bm{M}(y)\bm{u}(y)+\bm{f}(y)]dy.
\end{align*}

Denote by $T$ the operator that maps $\bm{u}$ to the left-hand side of
the above equation. For $\bm{\psi}\in C^{\infty}(\mathbb R^2)^2$, by
the similar arguments as those in the proof of Lemma 4.3 in \cite{LHL}, we
obtain
\begin{eqnarray*}
T\bm{\psi}(x)=-\bm{\psi}(x)+\int_{\partial B_r}[\bm{G}(x,y,\omega)
P\bm{\psi}(y)-P\bm{G}(x,y,\omega)\bm{\psi}(y)]ds(y),
\end{eqnarray*}
where $P\bm{\psi}:=\mu\frac{\partial
\bm{\psi}}{\partial\bm{\nu}}+(\lambda+\mu)(\nabla\cdot\bm{\psi})\bm{\nu}$ and
$\bm{\nu}$ is the unit normal vector on the boundary $\partial B_r$.

Approximating $\bm{u}$ with smooth functions, we get
\begin{align*}
-\bm{u}(x)+\int_{\partial B_r}[\bm{G}(x,y,\omega) P\bm{u}(y)-P\bm{G}(x,y,\omega)
\bm{u}(y)]ds(y)\\
=\int_{B_r}\bm{G}(x,y,\omega) [\bm{M}(y)\bm{u}(y)+\bm{f}(y)]dy.
\end{align*}
Using the radiation condition yields
\begin{eqnarray*}
\lim_{r\rightarrow\infty}\int_{\partial B_r}[\bm{G}(x,y,\omega)
P\bm{u}(y)-P\bm{G}(x,y,\omega) \bm{u}(y)]ds(y)=0.
\end{eqnarray*}
Therefore,
\begin{eqnarray*}
\bm{u}(x)+\int_D\bm{G}(x,y,\omega)
\bm{M}(y)\bm{u}(y)dy=-\int_D\bm{G}(x,y,\omega) \bm{f}(y)dy,\quad x\in\mathbb
R^2,
\end{eqnarray*}
which shows that $\bm{u}$ satisfies the Lippmann--Schwinger equation (\ref{c2})
and completes the proof. 
\end{proof}

The Lippmann--Schwinger equation (\ref{c2}) can be written in the operator form 
\begin{eqnarray}\label{c6}
(I+K_{\omega})\bm{u}=-H_{\omega}\bm{f},
\end{eqnarray}
where the operators $H_{\omega}$ and $K_{\omega}$ are defined by
\begin{align}\label{c7}
(H_{\omega}\bm{f})(x)&=\int_{D}\bm{G}(x,y,\omega)\bm{f}(y)dy,\quad x\in
D,\\\label{c8}
(K_{\omega}\bm{u})(x)&=\int_{D}\bm{G}(x,y,\omega)
\bm{M}(y)\bm{u}(y)dy,\quad x\in D.
\end{align}

\begin{lemma}\label{lemma5}
Assume that $p\geq 2$, $\frac{1}{p}+\frac{1}{p'}=1$,
$0<\varepsilon<\frac{2}{p}$, and $\bm{M}$ satisfies Assumption \ref{assu3}. Then
the operators $H_{\omega}: H_0^{-s}(D)^2\rightarrow H^{s}(D)^2$ and $K_{\omega}:
H^{\varepsilon,p}(D)^2\rightarrow H^{\varepsilon,p}(D)^2$ are bounded for $s\in
(0,1)$. Moreover, $K_{\omega}: H^{\varepsilon,p}(D)^2\rightarrow
H^{\varepsilon,p}(D)^2$ is compact.
\end{lemma}

\begin{proof}
We study the asymptotic expansion of the Green tensor $\bm{G}(x,y,\omega)$
when $|x-y|\to 0$. Recall the Green tensor:
\begin{eqnarray*}
\bm{G}(x,y,\omega)=\frac{1}{\mu}\Phi(x,y,\kappa_{\rm
s})\bm{I}+\frac{1}{\omega^2}\nabla_x\nabla_x^\top(\Phi(x,y,
\kappa_{\rm s})-\Phi(x,y,\kappa_{\rm p}))
\end{eqnarray*}
and the recurrence relation for the Hankel function of the first
kind \cite[(5.6.3)]{NN}:
\begin{eqnarray*}
\frac{d}{dt}[t^{-n}H_n^{(1)}(t)]=-t^{-n}H_{n+1}^{(1)}(t). 
\end{eqnarray*}
A direct calculation shows for $i,j=1,2$ that 
\begin{align}\notag
&\partial^2_{x_i x_j}[\Phi(x,y,\kappa_{\rm
s})-\Phi(x,y,\kappa_{\rm p})]\\\label{c9}
&=-\frac{\rm i}{4}\frac{1}{|x-y|}\Gamma_1(x-y,\omega)\delta_{ij}+\frac{\rm
i}{4}\frac{(x_i-y_i)(x_j-y_j)}{|x-y|^2}\Gamma_1(x-y,\omega),
\end{align}
where $\delta_{ij}$ is the Kronecker delta function. 
Substituting (\ref{b12})--(\ref{b13}) into (\ref{c9}) gives 
\begin{align}\label{c10}
&\partial^2_{x_i x_j}[\Phi(x,y,\kappa_{\rm
s})-\Phi(x,y,\kappa_{\rm p})]\notag\\
&=\frac{1}{4\pi}\left(\kappa_{\rm s}^2\ln\frac{\kappa_{\rm s}|x-y|}{2}
-\kappa_{\rm p}^2\ln\frac{\kappa_{\rm
p}|x-y|}{2}\right)\delta_{ij}+O(1).
\end{align}
Comparing (\ref{c10}) with (\ref{b8}), we conclude that the singularity of
$\nabla_x\nabla_x^\top(\Phi(x,y, \kappa_{\rm s})-\Phi(x,y,\kappa_{\rm p}))$ is
not exceeding the singularity of $\Phi(x,y,\kappa_{\rm s})\bm{I}$ when
$|x-y|\to 0$. It follows from Lemma \ref{lemma1} that $H_{\omega}:
H_0^{-s}(D)^2\rightarrow H^{s}(D)^2$ is bounded for $s\in (0,1)$.

For $\bm{u}\in H^{\varepsilon,p}(D)^2$ and $M_{ij}\in C_0^1(D) \subset
H_0^{-\varepsilon,p_1'}(D)$, by Lemma \ref{lemma1}, we obtain that
$M_{ij}\bm{u}$ is a well-defined element of $H_0^{-\varepsilon,p'}(D)^2$
and 
\begin{eqnarray}\label{c11}
\|M_{ij}\bm{u}\|_{H_0^{-\varepsilon,p'}(D)^2}\lesssim
\|M_{ij}\|_{H_0^{-\varepsilon,p_1'}(D)}\|\bm{u}\|_{H^{\varepsilon,p}(D)^2}.
\end{eqnarray}
For some fixed $\varepsilon\in (0,\frac{2}{p})$, we define
$\tilde{\delta}=\frac{1}{p}-\frac{\varepsilon}{2}\in (0,1)$ and
$s=1-\tilde{\delta}\in (0,1)$.  It is clear to note that
$\frac{1}{2}-\frac{s}{2}<\frac{1}{p}-\frac{\varepsilon}{2}$. The Sobolev
embedding theorem implies that $H^{s}(D)$ is embedded compactly into
$H^{\varepsilon,p}(D)$ and $H_0^{-\varepsilon,p'}(D)$ is embedded compactly into
$H^{-s}_0(D)$. Noting that $K_{\omega}\bm{u}=H_{\omega}(\bm{M}\bm{u})$ and
$\bm{M}\bm{u}\in H_0^{-\varepsilon,p'}(D)^2$, which is embedded compactly into
$H_0^{-s}(D)^2$, and that  $H_{\omega}: H_0^{-s}(D)^2\rightarrow
H^{s}(D)^2$ is bounded, we claim from (\ref{c11}) that $K_{\omega}:
H^{\varepsilon,p}(D)^2\rightarrow H^{\varepsilon,p}(D)^2$ is bounded and
compact.
\end{proof}

Now we present the existence of a unique solution of the direct scattering
problem (\ref{a1})--(\ref{a2}).

\begin{theorem}\label{theorem3}
Let $\bm{f}\in H_0^{-\varepsilon, p'}(D)^2$ with $0<\varepsilon<\frac{2}{p}$ and
$\bm{M}$ satisfy Assumption \ref{assu3}. Then the Lippmann--Schwinger equation
(\ref{c6}) has a unique solution $\bm{u}\in H_{\rm loc}^{\varepsilon,p}(\mathbb
R^2)^2$, which implies that the scattering problem (\ref{a1})--(\ref{a2}) has a
unique solution $\bm{u}\in H_{\rm loc}^{\varepsilon,p}(\mathbb R^2)^2$ which
satisfies the stability estimate 
\begin{eqnarray*}
\|\bm{u}\|_{H_{\rm loc}^{\varepsilon,p}(\mathbb R^2)^2}\lesssim\|\bm{f}\|_{
H_0^{-\varepsilon,p'}(\mathbb R^2)^2}.
\end{eqnarray*}
\end{theorem}

\begin{proof}
For the Lippmann--Schwinger equation $(I+K_{\omega})\bm{u}=-H_{\omega}\bm{f}$,
by Lemma \ref{lemma5}, we obtain that $H_{\omega}\bm{f}\in H^{\varepsilon,
p}(D)^2$ for $\bm{f}\in H_0^{-\varepsilon, p'}(D)^2$ and $I+K_{\omega}: 
H^{\varepsilon, p}(D)^2\rightarrow  H^{\varepsilon, p}(D)^2$ is a Fredholm
operator. Thus, by the Fredholm alternative, it suffices to show that
$(I+K_{\omega})\bm{u}=0$ has only the trivial solution $\bm{u}=0$.

For $(I+K_{\omega})\bm{u}=0$, we have 
\begin{eqnarray*}\label{c12}
\bm{u}(x)=-\int_{D}\bm{G}(x,y,\omega) \bm{M}(y)\bm{u}(y)dy,\quad x\in\mathbb
R^2.
\end{eqnarray*}
Thus we have $\bm{u}$ is smooth in $\mathbb R^2\setminus\overline{D}$ and 
\begin{eqnarray*}
\hat{\bm{u}}(\xi)=-\big[(4\pi^2\mu|\xi|^2-\omega^2)\bm{I}
+4\pi^2(\lambda+\mu)\xi\cdot\xi^\top\big]^{-1}\widehat{\bm{M}\bm{u}}(\xi)
\end{eqnarray*}
which implies
\begin{eqnarray*}
\big[4\pi^2\mu|\xi|^2+4\pi^2(\lambda+\mu)\xi\cdot\xi^\top-\omega^2\big]
\hat{\bm{u}}(\xi)=-\widehat{\bm{M}\bm{u}}(\xi).
\end{eqnarray*}
Taking the inverse Fourier transform of the above equation yields 
\begin{eqnarray}\label{c13}
\mu\Delta\bm{u}+(\lambda+\mu)\nabla\nabla^\top\cdot\bm{u}+\omega^2\bm{u}=\bm{M}
\bm{u}\quad{\rm in}~ \mathbb R^2.
\end{eqnarray}
By the Helmholtz decomposition, there exists two scalar potential
functions $\psi_1$ and $\psi_2$ such that 
\begin{eqnarray}\label{c14} 
\bm{u}=\nabla\psi_1+{\bf curl}\psi_2=(\partial_{x_1}\psi_1,
\partial_{x_2}\psi_1)^\top+ (\partial_{x_2}\psi_2, -\partial_{x_1}\psi_2)^\top.
\end{eqnarray}
 Substituting (\ref{c14}) into (\ref{c13}) gives that 
\begin{eqnarray*}\label{c15}
\nabla[(\lambda+2\mu)\Delta\psi_1+\omega^2\psi_1]+{\bf
curl}[\mu\Delta\psi_2+\omega^2\psi_2]=\bm{M}\nabla\psi_1+\bm{M}{\bf
curl}\psi_2\quad{\rm in}~\mathbb R^2,
\end{eqnarray*} 
which implies that
\begin{align*}
(\lambda+2\mu)\Delta(\nabla\psi_1)+\omega^2(\nabla\psi_1)=\bm{M}
\nabla\psi_1,\\
\mu\Delta({\bf curl}\psi_2)+\omega^2({\bf curl}\psi_2)=\bm{M}{\bf
curl}\psi_2.
\end{align*}

Letting $\bm{u}_{\rm p}=\nabla\psi_1$ and $\bm{u}_{\rm s}={\bf curl}\psi_2$,
we obtain that 
\begin{equation}\label{c16}
\begin{cases}
\Delta \bm{u}_{\rm p}+\kappa_{\rm p}^2\bm{u}_{\rm
p}=\frac{1}{\lambda+2\mu}\bm{M}\bm{u}_{\rm p}\quad{\rm in}~ \mathbb
R^2\\
\lim\limits_{r\rightarrow\infty}r^{\frac{1}{2}}\left(\partial_r
\bm{u}_{\rm p} -{\rm i}\kappa_{\rm p}\bm{u}_{\rm p}\right)=0
\end{cases}
\end{equation}
and
\begin{equation}\label{c18}
\begin{cases}
\Delta \bm{u}_{\rm s}+\kappa_{\rm s}^2\bm{u}_{\rm
s}=\frac{1}{\mu}\bm{M}\bm{u}_{\rm s}\quad{\rm in}~ \mathbb R^2\\
\lim\limits_{r\rightarrow\infty}r^{\frac{1}{2}}\left(\partial_r
\bm{u}_{\rm s} -{\rm i}\kappa_{\rm s}\bm{u}_{\rm s}\right)=0.
\end{cases}
\end{equation}
Since ${\rm supp}M_{ij}\subset D$, it follows from \eqref{c16}--\eqref{c18} that
$\bm{u}_{\rm p}$ and $\bm{u}_{\rm s}$ satisfy the homogeneous Helmholtz equation
in $\mathbb R^2\setminus\overline{D}$ and the Sommerfeld radiation condition.
Hence, $\bm{u}_{\rm p}$ and $\bm{u}_{\rm s}$ admit the following asymptotic
expansions
\begin{eqnarray}\label{c20}
\bm{u}_{\rm p}(x)=\frac{e^{{\rm i}\kappa_{\rm p}|x|}}{4\pi
|x|^{\frac{1}{2}}}\bm{u}_{{\rm p},\infty}(\hat{x})+o(|x|^{\frac{1}{2}}),\qquad
\bm{u}_{\rm s}(x)=\frac{e^{{\rm i}\kappa_{\rm s}|x|}}{4\pi
|x|^{\frac{1}{2}}}\bm{u}_{{\rm s},\infty}(\hat{x})+o(|x|^{\frac{1}{2}}).
\end{eqnarray}
Noting that $\bm{u}_{\rm p}$ satisfies the Sommerfeld radiation condition, when
$r\rightarrow \infty$, we have
\begin{eqnarray*}\label{c21}
\int_{\partial B_r}\left|\partial_r \bm{u}_{\rm p}-{\rm i}\kappa_{\rm p}
\bm{u}_{\rm p}\right|^2ds=\int_{\partial B_r}(|\partial_r \bm{u}_{\rm
p}|^2+\kappa_{\rm p}^2|\bm{u}_{\rm p}|^2)ds+2\kappa_{\rm p}{\rm
Im}\int_{\partial B_r}\bm{u}_{\rm p}\partial_{\nu}\overline{\bm {u}}_{\rm
p}ds\rightarrow 0.
\end{eqnarray*}
Combining the second Green theorem and (\ref{c16})--(\ref{c18}), we get 
\begin{align*}
&\int_{\partial B_r}\bm{u}_{\rm p}\partial_{\nu}\overline{\bm {u}}_{\rm
p}ds=\int_{B_r}|\nabla\bm{u}_{\rm p}|^2dx-\kappa_{\rm
p}^2\int_{B_r}|\bm{u}_{\rm p}|^2dx\\
&+\frac{1}{\lambda+2\mu}\int_{B_r}\left(M_{11}|u_{{\rm p},1}|^2+M_{22}|u_{{\rm
p},2}|^2+M_{12}u_{{\rm p},1}\overline{u_{{\rm p},2}}+M_{21}\overline{u_{{\rm
p},1}}u_{{\rm p},2}\right)dx,
\end{align*}
where $u_{{\rm p},1}$ and $u_{{\rm p},2}$ are the components of $\bm{u}_{\rm
p}$. Since $M$ is real-valued and symmetric, taking the imaginary part of the
above equation leads to ${\rm Im}\int_{\partial B_r}\bm{u}_{\rm
p}\partial_{\nu}\overline{\bm {u}}_{\rm p}ds=0$ which yields
$\lim\limits_{r\rightarrow \infty}\int_{\partial B_r}|\bm{u}_{\rm p}|^2dx=0$.
Using (\ref{c20}), we obtain $\int_{\partial B_1}|\bm{u}_{{\rm
p},\infty}|^2ds=0$, which implies $\bm{u}_{{\rm p},\infty}=0$, so $\bm{u}_{\rm
p}(x)=0$ in $\mathbb R^2\setminus\overline{D}$.  Similarly, we can obtain
$\bm{u}_{\rm s}=0$ in $\mathbb R^2\setminus\overline{D}$. Thus, we have
$\bm{u}=0$ in $\mathbb R^2\setminus\overline{D}$. Since $M_{ij}\in
C_0^1(\overline{D})$, it follows from the unique continuation (e.g.,
\cite{AITY}) that $\bm{u}=0$ in $\mathbb R^2$, which shows that $I+K_{\omega}$
is injective and completes the proof.  
\end{proof}

\section{Born approximation}

As shown in the previous section, the direct scattering problem is equivalent
to the Lippmann--Schwinger equation
\begin{eqnarray*}
\bm{u}(x)+\int_D\bm{G}(x,y,\omega)\bm{M}(y)\bm{u}(y)dy=-\int_D\bm{G}(x,y,\omega)
\bm{f}(y)dy,\quad x\in\mathbb R^2.
\end{eqnarray*}
Consider the Born sequence of the Lippmann--Schwinger equation 
\begin{eqnarray}\label{d1}
\bm{u}_n(x):=(-K_{\omega}\bm{u}_{n-1})(x),\quad n=1, 2, \dots, 
\end{eqnarray}
where the initial guess is given by 
\begin{equation*}
 \bm{u}_0(x):=(-H_{\omega}\bm{f})(x),
\end{equation*}
which is called the Born approximation to the solution of the 
Lippmann--Schwinger equation. Here, $K_{\omega}$ and $H_{\omega}$ are operators
given by (\ref{c7}) and (\ref{c8}), respectively.

We aim to show that for sufficient large $\omega$ and $x\in U$, the
Born series $\sum_{n=0}^{\infty}\bm{u}_n(x)$ converges to the solution
$\bm{u}(x)$ and the higher order terms decay in an appropriate way. 

\begin{lemma}\label{lemma6}
For any $1\leq p\leq 2\leq r\leq\infty$, $s\in(0,1)$ and $\omega\geq 1$, the
following estimates hold
\begin{align*}
\|H_{\omega}\|_{H_0^{-s,p}(D)^2\rightarrow H^{s,r}(D)^2} &\lesssim
\omega^{-1+2[s+(\frac{1}{p}-\frac{1}{r})]},\\
\|K_{\omega}\|_{H^{s,2p}(D)^2\rightarrow H^{s,2p}(D)^2} &\lesssim
\omega^{-1+2[s+(1-\frac{1}{p})]},\\
\|K_{\omega}\|_{H^{s,2p}(D)^2\rightarrow L^{\infty}(U)^2}&\lesssim
\omega^{1+2s-\frac{1}{p}},
\end{align*}
where the constant $c=c(\hat{\omega})$ is finite almost surely.
\end{lemma}

The proof of Lemma \ref{lemma6} can be found in \cite[Lemma 5]{LPS}. By Lemma
\ref{lemma6}, we have for large enough $\omega$ that 
\begin{eqnarray}\label{d5}
(I+K_{\omega})\sum_{n=0}^{N}\bm{u}_n=\bm{u}_0+(-1)^NK_{\omega}^{N+1}\bm{u}
_0\rightarrow \bm{u}_0\quad {\rm as}~ N\rightarrow \infty.
\end{eqnarray}
Since $(I+K_{\omega})^{-1}\bm{u}_0=\bm{u}$, taking the inverse of the operator
$I+K_{\omega}$ in (\ref{d5}) leads to
\begin{eqnarray}\label{d6}
\bm{u}(x,\omega)=\bm{u}_0(x,\omega)+\bm{u}_1(x,\omega)+\bm{b}(x,\omega),
\end{eqnarray}
where $\bm{b}(x,\omega):=\sum_{n=2}^{\infty}\bm{u}_n(x,\omega)$.
With the convergence of the Born approximation (\ref{d6}), we can analyze each
item in the Born approximation. For the leading item $\bm{u}_0$, we have
the following result \cite[Theorem 4.6]{LHL}.

\begin{theorem}\label{theorem4}
Let $\bm{f}$ satisfy Assumption \ref{assu1}. For all $x\in U$, it
holds almost surely that 
\begin{eqnarray*}\label{d7}
\lim_{Q\rightarrow\infty}\frac{1}{Q-1}\int_1^Q\omega^{m+1}|\bm{u}_0(x,
\omega)|^2 d\omega=a\int_{\mathbb R^2}\frac{1}{|x-y|}\phi(y)dy,
\end{eqnarray*}
where $a$ is a constant given in Theorem \ref{theorem1}. 
\end{theorem} 

Now we analyze the item $\bm{b}(x,\omega)$. For $n\geq 2$, by Lemma
\ref{lemma6}, we get
\begin{eqnarray*}
&&\|\bm{u}_n(x,\omega)\|_{L^{\infty}(U)^2}=\|K_{\omega}^n\bm{u}_0\|_{L^{\infty}
(U)^2}\\
&\leq &\|K_{\omega}\|_{H^{\varepsilon,p}(D)^2\rightarrow L^{\infty}(U)^2}
\|K_{\omega}\|^{n-1}_{H^{\varepsilon,p}(D)^2\rightarrow
H^{\varepsilon,p}(D)^2}\\
&&\quad \times\|H_{\omega}\|_{H_0^{-\varepsilon,p'}(D)^2\rightarrow
H^{\varepsilon,p}(D)^2}\|\bm{f}\|_{H_0^{-\varepsilon,p'}(D)^2}\\
&\lesssim &\omega^{1+2\varepsilon-\frac{2}{p}}\omega^{(n-1)[
-1+2(\varepsilon+1-\frac{2}{p})]}\omega^{-1+2[\varepsilon+\frac{1}{p'}-\frac{1}{
p}]}\\
&\lesssim
&\omega^{4\varepsilon+2-\frac{6}{p}}\omega^{(n-1)[-1+2(\varepsilon+\frac
{1}{p'}-\frac{1}{p})]},
\end{eqnarray*}
which gives
\begin{eqnarray*}\label{d9}
\sum_{n=2}^{\infty}\|\bm{u}_n\|_{L^{\infty}(U)^2}\lesssim
\omega^{4\varepsilon+2-\frac{6}{p}}\frac{\omega^{-1+2(\varepsilon+1-\frac{2}{p})
}}{1-\omega^{-1+2(\varepsilon+1-\frac{2}{p})}}\lesssim
\omega^{6\varepsilon+3-\frac{10}{p}}.
\end{eqnarray*}
Since $0<\varepsilon<\frac{2}{p}$ and $p>2$, we can choose suitable
$\varepsilon, p$ such that $\varepsilon'=6\varepsilon+5-\frac{10}{p}$ is 
small enough and 
\begin{eqnarray}\label{d10}
\sum_{n=2}^{\infty}||\bm{u}_n||_{L^{\infty}(U)^2}\lesssim
\omega^{-2+\varepsilon'}.
\end{eqnarray}
Hence, when $Q\to\infty$,
\begin{eqnarray}\label{d11}
\frac{1}{Q-1}\int_1^Q\omega^{m+1}|\bm{b}(x,\omega)|^2d\omega\lesssim
\frac{1}{Q-1}\int_1^Q\omega^{\alpha}d\omega=\frac{1}{\alpha+1}\frac{Q^{\alpha+1}
-1}{Q-1}\rightarrow 0,
\end{eqnarray}
where $\alpha=m+2\varepsilon'-3$. Note that $m\in [2,5/2)$, we have
$\alpha\in(-1,0)$ which is used in (\ref{d11}).

\section{The analysis of $\bm{u}_1(x,\omega)$}

In this section, we consider the term $\bm{u}_1(x,\omega)$ in the Born series
(\ref{d1}), which is given by. 
\begin{eqnarray*}\label{e1}
\bm{u}_1(x,\omega)=\int_D\int_D\bm{G}(x,y,\omega)\bm{M}(y)\bm{G}(y,z,\omega)\bm{
f}(z)dydz, \quad x\in U.
\end{eqnarray*}
It turns out the term $\bm u_1(x, \omega)$ is very difficult to analyze.
Fortunately, after tedious calculations, we find out that the contribution of
$\bm u_1$ can be ignored. We present the main result of this section. 

\begin{theorem}\label{theorem5}
Let $\bm{f}$, $U$, and $\bm{M}$ satisfy Assumption \ref{assu1}, Assumption
\ref{assu2}, and Assumption \ref{assu3}, respectively. Then for $x\in U$, it
holds almost surely that
\begin{eqnarray}\label{e2}
\lim_{Q\rightarrow\infty}\frac{1}{Q-1}\int_1^Q\omega^{m+1}|\bm{u}_1(x,
\omega)|^2 d\omega=0.
\end{eqnarray}
\end{theorem} 
\begin{proof}
Recall the Green tensor in (\ref{c1}), a direct computation shows 
\begin{align}\notag
\bm{G}(x,y,\omega)=&\left(\frac{{\rm i}}{4\mu}H_0^{(1)}(\kappa_{\rm
s}|x-y|)-\frac{{\rm
i}}{4\omega^2}\frac{1}{|x-y|}\Gamma_1(x-y,\omega)\right)\bm{I}\\\label{e3}
&+\frac{{\rm i}}{4\omega^2}\frac{1}{|x-y|^2}\Gamma_2(x-y,\omega)(x-y)\cdot
(x-y)^\top,
\end{align}
where $x-y=(x_1-y_1,x_2-y_2)^\top$ and $\Gamma_1,\Gamma_2$ are given in
(\ref{b12}), (\ref{b13}). Noting the definition of $H_{n,N}^{(1)}$ in
(\ref{b16}), we define the notations $\Theta_n(z,\omega):=\kappa_{\rm
s}^nH_{n,0}^{(1)}(\kappa_{\rm s}|z|)-\kappa_{\rm p}^nH_{n,0}^{(1)}(\kappa_{\rm
p}|z|)$, 
\begin{align}\notag
\bm{G}_0(x,y,\omega)=&\left(\frac{{\rm i}}{4\mu}H_{0,0}^{(1)}(\kappa_{\rm
s}|x-y|)-\frac{{\rm
i}}{4\omega^2}\frac{1}{|x-y|}\Theta_1(x-y,\omega)\right)\bm{I}\\\label{e4}
&+\frac{{\rm i}}{4\omega^2}\frac{1}{|x-y|^2}\Theta_2(x-y,\omega)(x-y)\cdot
(x-y)^\top,
\end{align}
and
\begin{eqnarray*}\label{e5}
\bm{u}_{1,l}(x,\omega):=\int_D\int_D\bm{G}_0(x,y,\omega)\bm{M}(y)\bm{G}(y,z,
\omega) \bm{f}(z)dydz,\quad x\in U.
\end{eqnarray*}

Now we estimate the order of the difference $\bm{u}_1-\bm{u}_{1,l}$ with respect
to the angular frequency $\omega$. A simple calculation yields 
\begin{eqnarray*}
&&|\bm{u}_1(x,\omega)-\bm{u}_{1,l}(x,\omega)|\\
&=&\left|\int_{D}\left(\bm{G}(x,y,\omega)-\bm{G}_0(x,y,\omega)\right)\bm{M}
(y)\int_D\bm{G}(y,z,\omega)\bm{f}(z)dzdy\right|\\
&\lesssim &\|\bm{G}(x,y,\omega)-\bm{G}_0(x,y,\omega)\|_{L^{p'}(D)^{2\times
2}}\|H_{\omega}\bm{f}\|_{L^p(D)^2}.
\end{eqnarray*}
Since $x\in U$, $y\in D$, there exists $c_1, c_2>0$ such that $c_1<|x-y|<c_2$.
By (\ref{b17}), we have 
\begin{eqnarray}\label{e6}
\|\Gamma_{n,0}(\kappa|x-\cdot|)\|_{L^{p'}(D)}\lesssim \kappa^{-\frac{3}{2}}.
\end{eqnarray} 
A direct computation shows that  $\nabla\Gamma_{n,0}(\kappa|x-\cdot|)\lesssim
\kappa^{-\frac{1}{2}}$. Hence
\begin{eqnarray}\label{e7}
\|\nabla\Gamma_{n,0}(\kappa|x-\cdot|)\|_{L^{p'}(D)}\lesssim
\kappa^{-\frac{1}{2}}.
\end{eqnarray} 
By (\ref{e6}) and (\ref{e7}), we get
\begin{eqnarray*}\label{e8}
\|\Gamma_{n,0}(\kappa|x-\cdot|)\|_{H^{\varepsilon,p'}(D)}\lesssim
\kappa^{-\frac{3}{2}+\varepsilon}.
\end{eqnarray*}
Therefore,
\begin{eqnarray}\label{e9}
\|\bm{G}(x,\cdot,\omega)-\bm{G}_0(x,\cdot,\omega)\|_{H^{\varepsilon,p'}(D)}
\lesssim  \omega^{-\frac{3}{2}+\varepsilon}.
\end{eqnarray}
It follows from Lemma \ref{lemma6} that we obtain 
\begin{eqnarray}\label{e10}
\|H_{\omega}\bm{f}\|_{H^{\varepsilon,p}(D)^2}\leq
\|H_{\omega}\|_{H_0^{-\varepsilon,p'}(D)^2\rightarrow
H^{\varepsilon,p}(D)^2}\|\bm{f}\|_{H_0^{-\varepsilon,p'}(D)^2}\lesssim
\omega^{-1+2(\varepsilon+1-\frac{2}{p})},
\end{eqnarray}
where we use the fact that $\|\bm{f}\|_{H_0^{-\varepsilon,p'}(D)^2}$ is
bounded almost surely. Denoting $\varepsilon_l=3\varepsilon+2(1-\frac{2}{p})$
which can be sufficient small for suitably chosen $\varepsilon$ and $p$ due to
$p\geq 2$ and $0<\varepsilon<\frac{2}{p}$, we have from (\ref{e9}) and
(\ref{e10}) that 
\begin{eqnarray}\label{e11}
|\bm{u}_1(x,\omega)-\bm{u}_{1,l}(x,\omega)|\lesssim
\omega^{-\frac{5}{2}+\varepsilon_l}.
\end{eqnarray}

In order to analyze the term $\bm{u}_{1,l}$, we replace the Green tensor
$\bm{G}(y,z,\omega)$ in $\bm{u}_{1,l}$ by $\bm{G}_0(y,z,\omega)$ and define
\begin{eqnarray*}\label{e12}
\bm{u}_{1,r}(x,\omega)=\int_D\int_D\bm{G}_0(x,y,\omega)\bm{M}(y)\bm{G}_0(y,z,\omega)\bm{f}(z)dydz,\quad x\in U.
\end{eqnarray*}
Next is estimate the order of the difference $\bm{u}_{1,l}-\bm{u}_{1,r}$
which is given by
\begin{align*}
\bm{u}_{1,l}(x,\omega)-\bm{u}_{1,r}(x,\omega)=&
\int_D\int_D\bm{G}_0(x,y,\omega)\bm{M}(y)\left(\bm{G}(y,z,\omega)-\bm{G}_0(y,z,
\omega)\right)\bm{f}(z)dydz\\
=&\Big(\sum_{j,k,l=1}^{2}I^{(1)}_{jkl},
\sum_{j,k,l=1}^{2}I^{(2)}_{jkl}\Big)^\top,
\end{align*}
where
\begin{eqnarray*}
I_{jkl}^{(i)}:=\int_D\int_DG_{0,ij}(x,y,\omega)M_{jk}(y)\left(G_{kl}(y,z,
\omega)-G_{0,kl}(y,z,\omega)\right)f_{l}(z)dydz
\end{eqnarray*}
for $i, j,k, l=1,2$. Here, $G_{ij}$ and $G_{0,ij}$ represent the elements of the
matrix $\bm{G}$ and $\bm{G}_0$, respectively. 

Now we only focus on the analysis of the term $I_{111}^{(1)}$ and show the
details, other terms can be analyzed in a similar way. In the dual sense, we
have
\begin{eqnarray}\label{e14}
I_{111}^{(1)} &=&\langle\bm{G}_{11}(y, z,\omega)-\bm{G}_{0,11}(y,
z,\omega), \notag\\
&&\quad \bm{G}_{0,11}(x,
y,\omega)M_{11}(y)f_1(z)\rangle_{(H^{\varepsilon,\tilde{p}}
(D\times D),H_0^{-\varepsilon,\tilde{p}'}(D\times D))}.
\end{eqnarray}
By (\ref{e3}) and (\ref{e4}), we can split
$\bm{G}_{11}(y,z,\omega)-\bm{G}_{0,11}(y,z,\omega)$ into three terms
\begin{eqnarray*}\label{e15}
\bm{G}_{11}(y,z,\omega)-\bm{G}_{0,11}(y,z,\omega)=g_0(y-z,\omega)+g_1(y-z,
\omega)+g_2(y-z,\omega),
\end{eqnarray*}
with 
\begin{eqnarray*}
g_0(y-z,\omega)&=&\frac{{\rm i}}{4\mu}\Gamma_{0,0}(\kappa_{\rm
s}|y-z|),\\
g_1(y-z,\omega)&=&-\frac{{\rm i}}{4\omega^2}\frac{1}{|y-z|}[\kappa_{\rm
s}\Gamma_{1,0}(\kappa_{\rm s}|y-z|)-\kappa_{\rm p}\Gamma_{1,0}(\kappa_{\rm
p}|y-z|)],\\
g_2(y-z,\omega)&=&\frac{{\rm
i}}{4\omega^2}\frac{(y_1-z_1)^2}{|y-z|^2}[\kappa_{\rm
s}^2\Gamma_{2,0}(\kappa_{\rm s}|y-z|)-\kappa_{\rm p}^2\Gamma_{2,0}(\kappa_{\rm
p}|y-z|)].
\end{eqnarray*}
Note $y,z\in D$ and $D$ is a bounded domain. Next is to estimate the term 
$\|\bm{G}_{11}(y,z,\omega)-\bm{G}_{0,11}(y,z,\omega)\|_{H^{\varepsilon,\tilde{p}
}(D\times D)}$, we only need to estimate
$\|g_j(z,\omega)\|_{H^{\varepsilon,\tilde{p}}(B)}, j=0,1,2$ for some bounded
domain containing the origin.

We analyze the three terms one by one. For large $\kappa_{\rm s}|z|$, it is
easy to note from (\ref{b17}) that
\begin{eqnarray}\label{e19}
|g_0(z,\omega)|\lesssim (\kappa_{\rm s}|z|)^{-\frac{3}{2}}.
\end{eqnarray}
For small $\kappa_{\rm s}|z|$, using (\ref{b8}) and (\ref{b16}) gives that
\begin{eqnarray}\label{e20}
|g_0(z,\omega)|\lesssim (\kappa_{\rm s}|z|)^{-\frac{1}{2}}=(\kappa_{\rm
s}|z|)^{-\frac{3}{2}}(\kappa_{\rm s}|z|)\lesssim(\kappa_{\rm
s}|z|)^{-\frac{3}{2}}.
\end{eqnarray}
By (\ref{e19}) and (\ref{e20}), we obtain
\begin{eqnarray}\label{e21}
\|g_0(z,\omega)\|^{\tilde{p}}_{L^{\tilde{p}}(B)}\lesssim\int_B\omega^{-\frac{3}{
2 }\tilde{p}}|z|^{-\frac{3}{2}\tilde{p}}dz\lesssim
\omega^{-\frac{3}{2}\tilde{p}}\int_0^Rr^{1-\frac{3}{2}\tilde{p}}dr\lesssim
\omega^{-\frac{3}{2}\tilde{p}},
\end{eqnarray}
holds for $\tilde{p}<\frac{4}{3}$, where $R:=\max\{|z|, z\in B\}$. Since
\begin{align*}
&\nabla g_0(z,\omega) = \frac{{\rm i}}{4\mu}\nabla\left(H_0^{(1)}(\kappa_{\rm
s}|z|)-a_0^{(0)}\sqrt{\frac{1}{\kappa_{\rm s}|z|}}e^{{\rm i}(\kappa_{\rm
s}|z|-\frac{\pi}{4})}\right)\\
=&\frac{{\rm i}}{4\mu}\frac{z}{|z|}\left(-\kappa_{\rm s}H_1^{(1)}(\kappa_{\rm
s}|z|)+\frac{1}{2}a_0^{(0)}\kappa_{\rm
s}^{-\frac{1}{2}}|z|^{-\frac{3}{2}}e^{{\rm i}(\kappa_{\rm
s}|z|-\frac{\pi}{4})}-{\rm i}a_0^{(0)}\kappa_{\rm
s}^{\frac{1}{2}}|z|^{-\frac{1}{2}}e^{{\rm i}(\kappa_{\rm
s}|z|-\frac{\pi}{4})}\right).
\end{align*}
Hence, we have for large $\kappa_{\rm s}|z|$ that 
\begin{eqnarray}\label{e23}
|\nabla g_0(z,\omega)|\lesssim \kappa_{\rm s}^{\frac{1}{2}}|z|^{-\frac{1}{2}}+
\kappa_{\rm s}^{-\frac{1}{2}}|z|^{-\frac{3}{2}}+ \kappa_{\rm
s}^{\frac{1}{2}}|z|^{-\frac{1}{2}}\lesssim  \kappa_{\rm
s}^{\frac{1}{2}}|z|^{-\frac{1}{2}}.
\end{eqnarray}
For small $\kappa_{\rm s}|z|$, we get
\begin{eqnarray}\label{e24}
|\nabla g_0(z,\omega)|\lesssim \kappa_{\rm s}(\kappa_{\rm s}|z|)^{-1}+
\kappa_{\rm s}^{-\frac{1}{2}}|z|^{-\frac{3}{2}}+ \kappa_{\rm
s}^{\frac{1}{2}}|z|^{-\frac{1}{2}}\lesssim  \kappa_{\rm
s}^{-\frac{1}{2}}|z|^{-\frac{3}{2}}.
\end{eqnarray}
By (\ref{e23}) and (\ref{e24}), we conclude for $\tilde{p}<\frac{4}{3}$ that
\begin{eqnarray}\label{e25}
\|\nabla g_0(z,\omega)\|^{\tilde{p}}_{L^{\tilde{p}}(B)}\lesssim\int_B\omega^{
\frac{1}{2} \tilde{p}}|z|^{-\frac{1}{2}\tilde{p}}dz+\int_B\omega^{-\frac{1}{2}
\tilde{p}}|z|^ {-\frac{3}{2}\tilde{p}}dz\lesssim
\omega^{\frac{1}{2}\tilde{p}}. 
\end{eqnarray}
Using (\ref{e21}) and (\ref{e25}), we have for
$\tilde{p}<\frac{4}{3}$ that
\begin{eqnarray}\label{e26}
\|g_0(z,\omega)\|_{H^{\varepsilon,\tilde{p}}(B)}\lesssim
\omega^{-\frac{3}{2}+2\varepsilon}. 
\end{eqnarray}
Now we analyze the term $g_1(z,\omega)$ which is given by
\begin{eqnarray*}\label{e27}
g_1(z,\omega)=-\frac{{\rm i}}{4\omega^2}\frac{1}{|z|}[\kappa_{\rm
s}\Gamma_{1,0}(\kappa_{\rm s}|z|)-\kappa_{\rm p}\Gamma_{1,0}(\kappa_{\rm
p}|z|)].
\end{eqnarray*}
For large $\omega |z|$, it follows from (\ref{b17}) that
\begin{eqnarray}\label{e28}
|g_1(z,\omega)|&\lesssim & \omega^{-2}|z|^{-1}[\omega(\omega
|z|)^{-\frac{3}{2}}]\lesssim
\omega^{-5/2}|z|^{-5/2}\notag\\
&=&\frac{(\omega|z|)^{-\frac{3}{2}}}{\omega|z|}
\lesssim (\omega|z|)^{-\frac{3}{2}}.
\end{eqnarray}
For small $\omega |z|$, by (\ref{b7}) and (\ref{b10}), we have 
\begin{eqnarray}\label{e29}
|g_1(z,\omega)|\lesssim (\omega|z|)^{-\frac{3}{2}}.
\end{eqnarray}
Combining (\ref{e28}) and (\ref{e29}) implies  for $\tilde{p}<\frac{4}{3}$ that
\begin{eqnarray*}\label{e30}
\|g_1(z,\omega)\|^{\tilde{p}}_{L^{\tilde{p}}(B)}\lesssim\int_B\omega^{-\frac{3}{
2}\tilde{p}}|z|^{-\frac{3}{2}\tilde{p}}dz\lesssim
\omega^{-\frac{3}{2}\tilde{p}}\int_0^Rr^{1-\frac{3}{2}\tilde{p}}dr\lesssim
\omega^{-\frac{3}{2}\tilde{p}}.
\end{eqnarray*}
For convenience, we split $g_1$ into two parts by
$g_1(z,\omega)=g_{11}(z,\omega)+g_{12}(z,\omega)$ with 
\begin{align*}
&g_{11}(z,\omega)=-\frac{{\rm i}}{4\omega^2}\frac{1}{|z|}\Gamma_1(z,\omega),\\
&g_{12}(z,\omega)=\frac{{\rm i}}{4\omega^2}\frac{1}{|z|}\Theta_1(z,\omega)
=-\frac{\rm i}{4}a_0^{(1)}e^{-\frac{3}{4}\pi{\rm i}}(c_{\rm
p}^{\frac{1}{2}}e^{{\rm i}\kappa_{\rm p}|z|}-c_{\rm s}^{\frac{1}{2}}e^{{\rm
i}\kappa_{\rm s}|z|})\omega^{-\frac{3}{2}}|z|^{-\frac{3}{2}}.
\end{align*}
For large $\omega |z|$, by (\ref{b9}), we have 
\begin{eqnarray}\label{e31}
|g_{11}(x,\omega)|\lesssim \omega^{-\frac{3}{2}}|z|^{-\frac{3}{2}}.
\end{eqnarray}
For small $\omega |z|$, by (\ref{b7}), we have
\begin{eqnarray}\label{e32}
|g_{11}(x,\omega)|\lesssim \left|\ln\frac{\omega|z|}{2}\right|\lesssim
\omega^{-\frac{3}{2}}|z|^{-\frac{3}{2}}.
\end{eqnarray}
Combining (\ref{e31}) and (\ref{e32}) yields for $\tilde{p}<\frac{4}{3}$ that 
\begin{eqnarray}\label{e33}
||g_{11}(z,\omega)||^{\tilde{p}}_{L^{\tilde{p}}(B)}\lesssim\int_B\omega^{-\frac{
3}{2}\tilde{p}}|z|^{-\frac{3}{2}\tilde{p}}dz\lesssim
\omega^{-\frac{3}{2}\tilde{p}}\int_0^Rr^{1-\frac{3}{2}\tilde{p}}dr\lesssim
\omega^{-\frac{3}{2}\tilde{p}}.
\end{eqnarray}
For the $\nabla g_{11}(z,\omega)$, we have
\begin{eqnarray*}
\nabla g_{11}(z,\omega)=\frac{\rm
i}{4\omega^2}\frac{z}{|z|^2}\Gamma_2(z,\omega).
\end{eqnarray*}
For large $\omega|z|$, (\ref{b9}) implies
\begin{eqnarray}\label{e34}
|\nabla g_{11}(z,\omega)|\lesssim  \omega^{-\frac{1}{2}}|z|^{-\frac{3}{2}}.
\end{eqnarray}
For small $\omega |z|$, (\ref{b8}) implies
\begin{eqnarray}\label{e35}
|\nabla g_{11}(z,\omega)|\lesssim
|z|^{-1}\lesssim\omega^{-\frac{1}{2}}|z|^{-\frac{3}{2}}.
\end{eqnarray}
Following (\ref{e34}) and (\ref{e35}), we get for $\tilde{p}<\frac{4}{3}$ that 
\begin{eqnarray}\label{e36}
\|\nabla g_{11}(z,\omega)\|^{\tilde{p}}_{L^{\tilde{p}}(B)}\lesssim\int_B\omega^{
-\frac{1} {2}\tilde{p}}|z|^{-\frac{3}{2}\tilde{p}}dz\lesssim
\omega^{-\frac{1}{2}\tilde{p}}\int_0^Rr^{1-\frac{3}{2}\tilde{p}}dr\lesssim
\omega^{-\frac{1}{2}\tilde{p}}.
\end{eqnarray}
Using (\ref{e33}) and (\ref{e36}), we have that
\begin{eqnarray}\label{e37}
\|g_{11}(z,\omega)\|_{H^{\varepsilon,\tilde{p}}(B)}\lesssim
\omega^{-\frac{3}{2}+\varepsilon}.
\end{eqnarray}

Since
\begin{eqnarray*}
g_{12}(z,\omega)=-\frac{\rm i}{4}a_0^{(1)}e^{-\frac{3}{4}\pi{\rm i}}(c_{\rm
p}^{\frac{1}{2}}e^{{\rm i}\kappa_{\rm p}|z|}-c_{\rm s}^{\frac{1}{2}}e^{{\rm
i}\kappa_{\rm s}|z|})\omega^{-\frac{3}{2}}|z|^{-\frac{3}{2}},
\end{eqnarray*}
it suffices to prove that $\omega^{-\frac{3}{2}}|z|^{-\frac{3}{2}}\in
H^{\varepsilon, \tilde{p}}(B)$. By the Slobodeckij semi-norm, we need
to prove 
\begin{eqnarray}\label{e38}
|\omega^{-\frac{3}{2}}|z|^{-\frac{3}{2}}|_{\varepsilon,\tilde{p},B}^{\tilde{p}}
=\omega^{-\frac{3}{2}\tilde{p}}\int_B\int_B\frac{||z|^{-\frac{3}{2}}-|z'|^{
-\frac{3}{2}}|^{\tilde{p}}}{|z-z'|^{2+\varepsilon\tilde{p}}}dzdz'< \infty.
\end{eqnarray}
To prove (\ref{e38}), we need the following two lemmas, one is the integrability
criterion and the other is Young's inequality for convolutions \cite[Theorem
2.24]{AF}.

\begin{lemma}\label{lemma7}
For the $n-$dimensional space, we have
\begin{eqnarray*}
\int_{|x|\leq 1}\frac{1}{|x|^{\rho}}dx<\infty \quad{\rm if\;\; and\;\; only\;\;
if}\quad \rho< n.
\end{eqnarray*}
\end{lemma}

This lemma is fundamental and can be easily proved by using the polar
coordinates. 

\begin{lemma}\label{lemma8}
Let $s_1,s_2,s_3\geq 1$ and suppose that
$\frac{1}{s_1}+\frac{1}{s_2}+\frac{1}{s_3}=2$. It holds that
\begin{eqnarray*}
\left|\int_{\mathbb R^n}\int_{\mathbb R^n}h_1(x)h_2(x-y)h_3(y)dxdy\right|\leq
\|h_1\|_{s_1}\|h_2\|_{s_2}\|h_3\|_{s_3},
\end{eqnarray*}
for any $h_1\in L^{s_1}(\mathbb R^n)$, $h_2\in L^{s_2}(\mathbb R^n)$, $h_3\in
L^{s_3}(\mathbb R^n)$.
\end{lemma}
Since
\begin{eqnarray*}
&&||z|^{-\frac{3}{2}}-|z'|^{-\frac{3}{2}}|=\left|\frac{(|z'|^{\frac{1}{2}}-|z|^{
\frac{1}{2}})(|z'|+|z'|^{\frac{1}{2}}|z|^{\frac{1}{2}}+|z|)}{|z|^{\frac{3}{2}}
|z'|^{\frac{3}{2}}}\right|\\
&&\leq\left|\frac{(|z'|-|z|)(|z'|^{\frac{1}{2}}+|z|^{\frac{1}{2}})^2}{|z|^{\frac
{3}{2}}|z'|^{\frac{3}{2}}(|z'|^{\frac{1}{2}}+|z|^{\frac{1}{2}})}\right|\leq
\frac{|z'-z|(|z'|^{\frac{1}{2}}+|z|^{\frac{1}{2}})}{|z|^{\frac{3}{2}}|z'|^{\frac
{3}{2}}},
\end{eqnarray*}
hence,
\begin{eqnarray*}
&&\int_B\int_B\frac{||z|^{-\frac{3}{2}}-|z'|^{-\frac{3}{2}}|^{\tilde{p}}}{
|z-z'|^{2+\varepsilon\tilde{p}}}dzdz'
\leq\int_B\int_B\frac{(|z'|^{\frac{1}{2}}+|z|^{\frac{1}{2}})^{\tilde{p}}}{|z|^{
\frac{3}{2}\tilde{p}}|z'|^{\frac{3}{2}\tilde{p}}|z-z'|^{2+\varepsilon\tilde{p}
-\tilde{p}}}dzdz'\\
&&\lesssim\int_B\int_B\frac{1}{|z|^{\tilde{p}}|z'|^{\frac{3}{2}\tilde{p}}|z-z'|^
{2+\varepsilon\tilde{p}-\tilde{p}}}dzdz'+\int_B\int_B\frac{1}{|z|^{\frac{3}{2}
\tilde{p}}|z'|^{\tilde{p}}|z-z'|^{2+\varepsilon\tilde{p}-\tilde{p}}}
dzdz'\\
&&:=I_1+I_2.
\end{eqnarray*}
We choose $\tilde{p}=\frac{10}{9}$, $\varepsilon=\frac{1}{5}$,
$s_1=\frac{89}{50}$, $s_2=\frac{59}{50}$ and $s_3=\frac{5251}{3102}$, then we
have $\frac{1}{s_1}+\frac{1}{s_2}+\frac{1}{s_3}=2$ and $\tilde{p}s_1<2$,
$\frac{3}{2}\tilde{p}s_2<2$, $(2+\varepsilon\tilde{p}-\tilde{p})s_3<2$. A direct
application of Lemmas \ref{lemma7} and \ref{lemma8} leads to 
\begin{eqnarray*}
I_1&=&\int_{\mathbb R^2}\int_{\mathbb
R^2}\chi_{B}(z)|z|^{-\tilde{p}}\chi_B(z')|z'|^{-\frac{3}{2}\tilde{p}}\chi_{B_{2R
}}(z-z')|z-z'|^{-(2+\varepsilon\tilde{p}-\tilde{p})}dzdz'\\
&\leq&
\||z|^{-\tilde{p}}\|_{s_1}\||z|^{-\frac{3}{2}\tilde{p}}\|_{s_2}\||z|^{
-(2+\varepsilon\tilde{p}-\tilde{p})}\|_{s_3}<\infty,
\end{eqnarray*}
where $B_{2R}$ is the ball with radius $2R$ and center at the origin, and
$\chi_B$ is the characteristic function of the domain $B$ which equals to 1 in
$B$ and vanishes outside of $B$. We can prove $I_2<\infty$ by a similar
argument. Therefore, 
\begin{eqnarray}\label{e39}
\|g_{12}(z,\omega)\|_{H^{\frac{1}{5},\frac{10}{9}}(B)}\lesssim
\omega^{-\frac{3}{2}},
\end{eqnarray}

Next we analyze the term $g_2(z,\omega)$ which is given by
\begin{eqnarray*}\label{e40}
g_2(z,\omega)=\frac{{\rm i}}{4\omega^2}\frac{z_1^2}{|z|^2}[\kappa_{\rm
s}^2\Gamma_{2,0}(\kappa_{\rm s}|z|)-\kappa_{\rm p}^2\Gamma_{2,0}(\kappa_{\rm
p}|z|)].
\end{eqnarray*}
For large $\omega|z|$,  (\ref{b17}) shows that
\begin{eqnarray}\label{e41}
|g_2(z,\omega)|\lesssim \frac{1}{\omega^2}\left(\kappa_{\rm s}^2(\kappa_{\rm
s}|z|)^{-\frac{3}{2}}+\kappa_{\rm p}^2(\kappa_{\rm
p}|z|)^{-\frac{3}{2}}\right)\lesssim (\omega|z|)^{-\frac{3}{2}}.
\end{eqnarray}
For small $\omega|z|$, from (\ref{b13}) we have 
\begin{eqnarray}\label{e42}
|g_2(z,\omega)|\lesssim (\kappa_{\rm s}|z|)^{-\frac{1}{2}}+ (\kappa_{\rm
p}|z|)^{-\frac{1}{2}}\lesssim  (\omega |z|)^{-\frac{1}{2}}\lesssim 
(\omega|z|)^{-\frac{3}{2}}.
\end{eqnarray}
Thus, (\ref{e41}) and (\ref{e42}) implies for $\tilde{p}<\frac{4}{3}$ that
\begin{eqnarray*}\label{e43}
||g_2(z,\omega)||^{\tilde{p}}_{L^{\tilde{p}}(B)}\lesssim\int_B\omega^{-\frac{3}{
2}\tilde{p}}|z|^{-\frac{3}{2}\tilde{p}}dz\lesssim
\omega^{-\frac{3}{2}\tilde{p}}\int_0^Rr^{1-\frac{3}{2}\tilde{p}}dr\lesssim
\omega^{-\frac{3}{2}\tilde{p}}.
\end{eqnarray*}
A direct computation shows that
\begin{eqnarray*}
&&\nabla g_2(z,\omega)=\frac{\rm
i}{2\omega^2}z_1\bm{e}_1a_0^{(2)}|z|^{-\frac{5}{2}}e^{-\frac{5}{4}\pi{\rm
i}}\left[\kappa_{\rm p}^{\frac{3}{2}}e^{{\rm i}\kappa_{\rm p}|z|}-\kappa_{\rm
s}^{\frac{3}{2}}e^{{\rm i}\kappa_{\rm s}|z|}\right]\\
&&\quad +\frac{\rm
i}{2\omega^2}\frac{z_1}{|z|^2}\bm{e}_1\Gamma_2(z,\omega)-\frac{\rm
i}{4\omega^2}\frac{z}{|z|}\frac{z_1^2}{|z|^2}\Gamma_3(z,\omega)
+\frac{\rm
i}{4\omega^2}a_0^{(2)}\frac{z_1^2}{|z|^2}\frac{z}{|z|}e^{-\frac{5}{4}\pi{\rm
i}}\times\\
&&\quad \left[\left(\frac{5}{2}\kappa_{\rm
s}^{\frac{3}{2}}|z|^{-\frac{3}{2}}-{\rm
i}\kappa_{\rm s}^{\frac{5}{2}}|z|^{-\frac{1}{2}}\right)e^{{\rm i}\kappa_{\rm
s}|z|}-\left(\frac{5}{2}\kappa_{\rm p}^{\frac{3}{2}}|z|^{-\frac{3}{2}}-{\rm
i}\kappa_{\rm p}^{\frac{5}{2}}|z|^{-\frac{1}{2}}\right)e^{{\rm i}\kappa_{\rm
p}|z|}\right].
\end{eqnarray*}
For large $\omega |z|$, from (\ref{b9}) we know
\begin{eqnarray}\label{e44}
|\nabla
g_2(z,\omega)|\lesssim\omega^{-\frac{1}{2}}|z|^{-\frac{3}{2}}+\omega^{\frac{1}{2
}}|z|^{-\frac{1}{2}}\lesssim \omega^{\frac{1}{2}}|z|^{-\frac{1}{2}}.
\end{eqnarray}
For small $\omega |z|$, from (\ref{b13}) and (\ref{b14}), we obtain 
\begin{eqnarray}\label{e45}
|\nabla
g_2(z,\omega)|\lesssim\omega^{-\frac{1}{2}}|z|^{-\frac{3}{2}}+\omega^{\frac{1}{2
}}|z|^{-\frac{1}{2}}+|z|^{-1}\lesssim \omega^{-\frac{1}{2}}|z|^{-\frac{3}{2}}.
\end{eqnarray}
By (\ref{e44}) and (\ref{e45}), we conclude for $\tilde{p}<\frac{4}{3}$ that
\begin{eqnarray*}\label{e46}
\|\nabla
g_2(z,\omega)\|^{\tilde{p}}_{L^{\tilde{p}}(B)}\lesssim\int_B\omega^{\frac{1}{2}
\tilde{p}}|z|^{-\frac{1}{2}\tilde{p}}dz+\int_B\omega^{-\frac{1}{2}\tilde{p}}|z|^
{-\frac{3}{2}\tilde{p}}dz\lesssim \omega^{\frac{1}{2}\tilde{p}}.
\end{eqnarray*}
Using (\ref{e44}) and (\ref{e47}) and interpolation, we have for
$\tilde{p}<\frac{4}{3}$ that
\begin{eqnarray}\label{e47}
\|g_2(z,\omega)\|_{H^{\varepsilon,\tilde{p}}(B)}\lesssim
\omega^{-\frac{3}{2}+2\varepsilon}.
\end{eqnarray}
Noting that $D$ is a bounded domain, and combining (\ref{e26}),
(\ref{e37}), (\ref{e39}), and (\ref{e47}), we obtain for any $\varepsilon\in
(0,\frac{1}{5}]$ and $\tilde{p}\in [1,\frac{10}{9}]$ that 
\begin{eqnarray*}
\|\bm{G}_{11}(y,z,\omega)-\bm{G}_{0,11}(y,z,\omega)\|_{H^{\varepsilon,\tilde{p}
}(D\times D)}\lesssim \omega^{-\frac{3}{2}+2\varepsilon}.
\end{eqnarray*}
Since $\bm{G}_{0,11}(x,y,\omega)$ is smooth for $x\in U$ and $y\in
D$, $M_{11}(y)\in C_0^1(\overline{D})$, and $f_1(z)\in
H^{-\varepsilon,\tilde{p}}(D)$ for any $\varepsilon>0$ and $1<\tilde{p}<\infty$,
we have $\bm{G}_{0,11}(x,y,\omega)M_{11}(y)f_1(z)\in
H_0^{-\varepsilon,\tilde{p}}(D\times D)$. Moreover, 
\begin{eqnarray*}
\bm{G}_{0,11}(x,y,\omega)&=&\frac{\rm i}{4\mu}\frac{e^{-\frac{\pi}{4}{\rm
i}}}{|x-y|^{\frac{1}{2}}}a_0^{(0)}c_{\rm s}^{-\frac{1}{2}}e^{{\rm i}\kappa_{\rm
s}|x-y|}\omega^{-\frac{1}{2}}-\frac{\rm i}{4}\frac{e^{-\frac{3}{4}\pi{\rm
i}}}{|x-y|^{\frac{3}{2}}}a_0^{(1)}\\
&&\times\left(c_{\rm s}^{\frac{1}{2}}e^{{\rm
i}\kappa_{\rm s}|x-y|}-c_{\rm p}^{\frac{1}{2}}e^{{\rm i}\kappa_{\rm
p}|x-y|}\right)\omega^{-\frac{3}{2}}\\
&&+\frac{\rm i}{4}\frac{e^{-\frac{5}{4}\pi{\rm
i}}(x_1-y_1)^2}{|x-y|^{5/2}}a_0^{(2)}\left(c_{\rm s}^{\frac{3}{2}}e^{{\rm
i}\kappa_{\rm s}|x-y|}-c_{\rm p}^{\frac{3}{2}}e^{{\rm i}\kappa_{\rm
p}|x-y|}\right)\omega^{-\frac{1}{2}}.
\end{eqnarray*}
Thus, we obtain for sufficient large $\omega$ that
\begin{eqnarray}\label{e48}
\|\bm{G}_{0,11}(x,y,\omega)M_{11}(y)f_1(z)\|_{H_0^{-\varepsilon,\tilde{p}}
(D\times D)}\lesssim \omega^{-\frac{1}{2}}.
\end{eqnarray}
Substituting (\ref{e47}) and (\ref{e48}) into (\ref{e14}) yields
$|I_{111}^{(1)}|\lesssim \omega^{-2+\varepsilon}$ holds for any $\varepsilon\in
(0,\frac{1}{5}]$. 

Using similar proofs, we can obtain estimates for
$I_{112}^{(1)}$, $\cdots$, $I_{222}^{(2)}$ and get 
\begin{eqnarray*}\label{e49}
|\bm{u}_{1,l}(x,\omega)-\bm{u}_{1,r}(x,\omega)|\lesssim \omega^{-2+\varepsilon}.
\end{eqnarray*}
Noting (\ref{e11}), we have 
\begin{eqnarray*}\label{e50}
|\bm{u}_{1}(x,\omega)-\bm{u}_{1,r}(x,\omega)|\lesssim \omega^{-2+\varepsilon}.
\end{eqnarray*}
Since 
\begin{eqnarray*}
\frac{1}{Q-1}\int_1^Q\omega^{m+1}|\bm{u}_1(x,\omega)|^2d\omega&\lesssim&\frac{2}
{Q-1}\int_1^Q\omega^{m+1}|\bm{u}_{1,r}(x,\omega)|^2d\omega\\
&&+\frac{2}{Q-1}
\int_1^Q\omega^{m-3+2\varepsilon}d\omega
\end{eqnarray*}
and 
\[
\frac{2}{Q-1}\int_1^Q\omega^{m-3+2\varepsilon}d\omega\to 0
\]
for $m\in [2,5/2)$ and small enough $\varepsilon$. To prove (\ref{e2}), it is
sufficient to prove 
\begin{eqnarray}\label{e51}
\lim_{Q\to\infty}\frac{1}{Q-1}\int_1^Q\omega^{m+1}|\bm{u}_{1,r}(x,
\omega)|^2d\omega=0,\quad x\in U. 
\end{eqnarray}

It follows from a straightforward but tedious calculation that the
vector $\bm{u}_{1,r}(x,\omega)$ can be decomposed into three parts according to
the order of $\omega$ in the following form
\begin{eqnarray}\label{e52}
\bm{u}_{1,r}(x,\omega)=\bm{v}_1(x,\omega)\omega^{-1}+\bm{v}_2(x,\omega)\omega^{
-2}+\bm{v}_3(x,\omega)\omega^{-3},
\end{eqnarray}
where
\begin{align*}
&\bm{v}_1(x,\omega)=-\frac{e^{-\frac{\pi}{2}{\rm i}}}{16\mu^2c_{\rm
s}}{a^{(0)}_0}^2\int_D\int_D e^{{\rm i}\kappa_{\rm
s}(|x-y|+|y-z|)}\frac{\bm{M}(y)\bm{f}(z)}{|x-y|^{\frac{1}{2}}|y-z|^{\frac{1}{2}}
}dydz
-\frac{e^{-\frac{3}{2}\pi{\rm i}}}{16\mu}a_0^{(0)}a_0^{(2)}\\
&\quad\times\bigg[\int_D\int_D\left(c_{\rm s}e^{{\rm i}\kappa_{\rm
s}(|x-y|+|y-z|)}-c_{\rm p}^{\frac{3}{2}}c_{\rm s}^{-\frac{1}{2}}e^{{\rm
i}(\kappa_{\rm s}|x-y|+\kappa_{\rm
p}|y-z|)}\right)\frac{\bm{M}(y)\bm{J}(y-z)\bm{f}(z)}{|x-y|^{\frac{1}{2}}|y-z|^{
\frac{5}{2}}}dydz\\
&\quad+\int_D\int_D\left(c_{\rm s}e^{{\rm i}\kappa_{\rm s}(|x-y|+|y-z|)}-c_{\rm
p}^{\frac{3}{2}}c_{\rm s}^{-\frac{1}{2}}e^{{\rm i}(\kappa_{\rm
s}|y-z|+\kappa_{\rm
p}|x-y|)}\right)\frac{\bm{J}(x-y)\bm{M}(y)\bm{f}(z)}{|x-y|^{\frac{5}{2}}|y-z|^{
\frac{1}{2}}}dydz\bigg]\\
&\quad+\frac{e^{-\frac{5}{2}\pi{\rm
i}}}{16}{a_0^{(2)}}^2\int_D\int_D\bigg(-c_{\rm
s}^3e^{{\rm i}\kappa_{\rm s}(|x-y|+|y-z|)}+c_{\rm s}^{\frac{3}{2}}c_{\rm
p}^{\frac{3}{2}}e^{{\rm i}(\kappa_{\rm s}|x-y|+\kappa_{\rm p}|y-z|)}\\
&\quad+c_{\rm s}^{\frac{3}{2}}c_{\rm p}^{\frac{3}{2}}e^{{\rm
i}(\kappa_{\rm p}|x-y|+\kappa_{\rm s}|y-z|)}-c_{\rm p}^3e^{{\rm i}\kappa_{\rm
p}(|x-y|+|y-z|)}\bigg) \frac{\bm{J}(x-y)\bm{M}(y)\bm{J}(y-z)\bm{f}(z)}{|x-y|^{
\frac{5}{2}}|y-z|^{\frac{5}{2}}}dydz,
\end{align*}
\begin{align*}
&\bm{v}_2(x,\omega)=\frac{e^{-\pi{\rm
i}}}{16\mu}a_0^{(0)}a_0^{(1)}\int_D\int_D\bigg(e^{{\rm i}\kappa_{\rm
s}(|x-y|+|y-z|)}-c_{\rm p}^{\frac{1}{2}}c_{\rm s}^{-\frac{1}{2}}e^{{\rm
i}(\kappa_{\rm s}|x-y|+\kappa_{\rm
p}|y-z|)}\bigg)\\
&\quad\times \frac{\bm{M}(y)\bm{f}(z)}{|x-y|^{\frac{1}{2}}|y-z|^{\frac{3}{2}
}}dydz+\frac{e^{-\pi{\rm
i}}}{16\mu}a_0^{(0)}a_0^{(1)}\int_D\int_D\bigg(e^{{\rm
i}\kappa_{\rm s}(|x-y|+|y-z|)}\\
&\quad -c_{\rm p}^{\frac{1}{2}}c_{\rm s}^{-\frac{1}{2}}e^{{\rm i}(\kappa_{\rm
p}|x-y|+\kappa_{\rm
s}|y-z|)}\bigg)\frac{\bm{M}(y)\bm{f}(z)}{|x-y|^{\frac{3}{2}}|y-z|^{\frac{1}{2}}}
dydz+\frac{e^{-2\pi{\rm i}}}{16}a_0^{(1)}a_0^{(2)}\\
&\quad \times\int_D\int_D\bigg(c_{\rm
s}^2e^{{\rm i}\kappa_{\rm s}(|x-y|+|y-z|)}
-c_{\rm s}^{\frac{1}{2}}c_{\rm p}^{\frac{3}{2}}
e^{{\rm i}(\kappa_{\rm s}|x-y|+\kappa_{\rm p}|y-z|)}\\
&\quad -c_{\rm p}^{\frac{1}{2}}c_{\rm s}^{\frac{3}{2}}e^{{\rm
i}(\kappa_{\rm p}|x-y|+\kappa_{\rm s}|y-z|)}+c_{\rm p}^2e^{{\rm i}\kappa_{\rm
p}(|x-y|+|y-z|)}\bigg)\frac{\bm{M}(y)\bm{J}(y-z)\bm{f}(z)}{|x-y|^{\frac{3}{2}}
|y-z|^{\frac{5}{2}}}dydz\\
&\quad+\frac{e^{-2\pi{\rm i}}}{16}a_0^{(1)}a_0^{(2)}\int_D\int_D\bigg(c_{\rm
s}^2e^{{\rm i}\kappa_{\rm s}(|x-y|+|y-z|)}
-c_{\rm s}^{\frac{1}{2}}c_{\rm p}^{\frac{3}{2}}
e^{{\rm i}(\kappa_{\rm p}|x-y|+\kappa_{\rm s}|y-z|)}\\
&\quad-c_{\rm p}^{\frac{1}{2}}c_{\rm s}^{\frac{3}{2}}e^{{\rm
i}(\kappa_{\rm s}|x-y|+\kappa_{\rm p}|y-z|)}+c_{\rm p}^2e^{{\rm i}\kappa_{\rm
p}(|x-y|+|y-z|)}\bigg)\frac{\bm{J}(x-y)\bm{M}(y)\bm{f}(z)}{|x-y|^{\frac{5}{2}}
|y-z|^{\frac{3}{2}}}dydz,
\end{align*}
\begin{align*}
&\bm{v}_3(x,\omega)=\frac{e^{-\frac{3}{2}\pi{\rm
i}}}{16}{a_0^{(1)}}^2\int_D\int_D\bigg(-c_{\rm s}e^{{\rm i}\kappa_{\rm
s}(|x-y|+|y-z|)}+c_{\rm s}^{\frac{1}{2}}c_{\rm p}^{\frac{1}{2}}e^{{\rm
i}(\kappa_{\rm s}|x-y|+\kappa_{\rm p}|y-z|)}\\
&\quad+c_{\rm s}^{\frac{1}{2}}c_{\rm p}^{\frac{1}{2}}e^{{\rm
i}(\kappa_{\rm p}|x-y|+\kappa_{\rm s}|y-z|)}-c_{\rm p}e^{{\rm i}\kappa_{\rm
p}(|x-y|+|y-z|)}\bigg)\frac{\bm{M}(y)\bm{f}(z)}{|x-y|^{\frac{3}{2}}|y-z|^{\frac{
3}{2}}}dydz.
\end{align*}
Here $\bm{J}(x-y)=(x-y)(x-y)^\top$ and  $\bm{J}(y-z)=(y-z)(y-z)^\top$.

By (\ref{e52}) and the Cauchy--Schwartz inequality, we have 
\begin{eqnarray*}
\int_1^Q\omega^{m+1}|\bm{u}_{1,r}(x,\omega)|^2d\omega &\lesssim&
\int_1^Q\big(\omega^{m-1}|\bm{v}_1(x,\omega)|^2\\
&&+\omega^{m-3}|\bm{v}_2(x,
\omega)|^2+\omega^{m-5}|\bm{v}_3(x,\omega)|^2\big)d\omega.
\end{eqnarray*}
Noting the facts that $|x-y|$ has a positive lower bound for $x\in U$, $y\in D$,
$\||y-z|^{-\frac{3}{2}}\|_{H^{\frac{1}{5},\frac{10}{9}}(D\times D)}$ is bounded
from the above analysis about $g_{12}$, $M_{ij}(y)\in C_0^{1}(\overline{D})$,
and $\|f_j(z)\|_{H^{-\frac{1}{5},10}(D)}$ is bounded from the
assumption, we conclude that 
\begin{eqnarray*}
|\bm{v}_2(x,\omega)|<\infty, \quad |\bm{v}_3(x,\omega)|<\infty,\quad x\in
U, ~ \omega\geq 1.
\end{eqnarray*}
Hence, we have as $\omega\to\infty$ that 
\begin{align*}
&\frac{1}{Q-1}\int_1^Q\omega^{m-3}|\bm{v}_2(x,\omega)|^2d\omega\lesssim
\frac{1}{Q-1}\int_1^Q\omega^{m-3}d\omega\to 0,\\
&\frac{1}{Q-1}\int_1^Q\omega^{m-5}|\bm{v}_3(x,\omega)|^2d\omega\lesssim
\frac{1}{Q-1}\int_1^Q\omega^{m-5}d\omega\to 0.
\end{align*}
To prove (\ref{e51}), it suffices to prove that 
\begin{eqnarray}\label{e53}
\lim_{Q\to\infty}\frac{1}{Q-1}\int_1^Q\omega^{m-1}|\bm{v}_1(x,
\omega)|^2d\omega=0.
\end{eqnarray}
We claim that in order to prove (\ref{e53}), it will be enough to show that 
\begin{eqnarray}\label{e54}
\int_1^{\infty}\omega^{m-2}|\bm{v}_1(x,\omega)|^2d\omega<\infty,\quad
\text{almost surely}.
\end{eqnarray}
To show this, we notice that 
\begin{eqnarray*}
\frac{1}{Q}\int_1^Q\omega^{m-1}|\bm{v}_1(x,\omega)|^2d\omega&\leq&
\int_1^Q\frac{\omega}{Q}\omega^{m-2}|\bm{v}_1(x,\omega)|^2d\omega\\
&\leq&\int_1^{\infty}\min(1,\frac{\omega}{Q})\omega^{m-2}|\bm{v}_1(x,
\omega)|^2d\omega.
\end{eqnarray*}
From the dominated convergence theorem, the last integral in the above
inequality converges almost surely to zero as $Q\to\infty$, so the claim
follows. The remaining part of the proof will focus on (\ref{e54}). To this end,
we define 
\begin{eqnarray}\label{e55}
g(x,\omega)&=&\int_D\int_De^{{\rm
i}\omega(c_1|x-y|+c_2|y-z|)}\notag\\
&&\frac{(x_1-y_1)^{p_1}(x_2-y_2)^{p_2} (y_1-z_1)^ { p_3 }
(y_2-z_2)^{p_4}}{|x-y|^{l_1}|y-z|^{l_2}}q(y)\tilde{f}(z)dydz
\end{eqnarray}
where $c_1,c_2>0$, $p_1,...,l_2\geq 0$, $\tilde{f}$ denotes a generalized
Gaussian random field which equals to $f_1$ or $f_2$, and $q(y)\in
C_0^1(\overline{D})$ stands for $M_{ij}(y)$. From the formulation of
$\bm{v}_1(x,\omega)$,  we know that it is a linear combination of $g(x,\omega)$
for different $(l_1,l_2,p_1,p_2,p_3,p_4)\in S$ which is given by
\begin{align*}
S=&\Big\{(\frac{1}{2},\frac{1}{2},0,0,0,0),\;(\frac{1}{2},\frac{5}{2},0,0,2,0),
\;(\frac{1}{2},\frac{5}{2},0,0,1,1),\;(\frac{1}{2},\frac{5}{2},0,0,0,2),\\
&\;\;(\frac{5}{2},\frac{1}{2},2,0,0,0),\;(\frac{5}{2},\frac{1}{2},0,2,0,0),
\;(\frac{5}{2},\frac{1}{2},1,1,0,0),
(\frac{5}{2},\frac{5}{2},2,0,2,0),\\
&\;\;(\frac{5}{2},\frac{5}{2},1,1,1,1),\;(\frac{5}{2},\frac{5}{2},2,0,1,1),
\;(\frac{5}{2},\frac{5}{2},1,1,0,2),\;(\frac{5}{2},\frac{5}{2},2,0,0,2),\\
&\;\;(\frac{5}{2},\frac{5}{2},1,1,2,0),\;(\frac{5}{2},\frac{5}{2},0,2,2,0),
\;(\frac{5}{2},\frac{5}{2},0,2,1,1),\;(\frac{5}{2},\frac{5}{2},0,2,0,2)\Big\}. 
\end{align*}
To prove (\ref{e54}), it is enough to show that 
\begin{eqnarray}\label{e56}
\int_1^{\infty}\omega^{m-2}|g(x,\omega)|^2d\omega<\infty,\quad
\text{almost surely}.
\end{eqnarray}

In the following, we consider two cases. 

Case 1. $m=2$. In this case, Lemma \ref{lemma3} claims that
$\tilde{f}\in H^{-\varepsilon,p}(D)$ almost surely for any $\varepsilon>0$ and
$1<p<\infty$. In order to avoid the distribution dualities, we introduce the
modification $\tilde{f}_{\delta}:=\tilde{f}*\rho_{\delta}$ where
$\rho_{\delta}:=\delta^{-2}\rho\left(\frac{x}{\delta}\right)$,$\rho\in
C_0^{\infty}(\mathbb R^2)$ is a radially symmetric function satisfying
$\int_{\mathbb R^2}\rho(x)dx=1$. We denote $g_{\delta}$ by replacing $\tilde{f}$
by the standard mollification $\tilde{f}_{\delta}$ in (\ref{e55}). Let
$M_{\delta}\tilde{f}:=\tilde{f}_{\delta}$ be the modification operator, and
$C_{\delta}$ be the covariance operator of $\tilde{f}_{\delta}$. Then it is
easy to verify that $C_{\delta}=M_{\delta}C_{\tilde{f}}M_{\delta}$ and
$g_{\delta}(x,\omega)\to g(x,\omega)$ as $\delta\to 0$. To prove (\ref{e56}), we
claim that it is enough to show that
\begin{eqnarray}\label{e57}
\sup_{\delta\in(0,1)}\int_1^{\infty}{\mathbb
E}|g_{\delta}(x,\omega)|^2d\omega<\infty.
\end{eqnarray}
If (\ref{e57}) holds, applying the Fubini theorem and Fatou lemma implies that
\[
{\mathbb E}\left(\int_1^{\infty}|g(x,\omega)|^2d\omega\right)<\infty,
\]
which
shows that (\ref{e56}) holds immediately. So, we focus on the prove of
(\ref{e57}) for this case. To this end, we look at the phase function
$A(y,z)=c_1|x-y|+c_2|y-z|$ for some fixed $x\in U$. It is easy to see that
$A(y,z)$ is smooth on $D\times D$ apart from the subset where $y=z$. A direct
computation shows 
\begin{eqnarray*}
\nabla_y A(y,z)=c_1\frac{y-x}{|y-x|}+c_2\frac{y-z}{|y-z|}, \quad \nabla_z
A(y,z)=c_2\frac{z-y}{|z-y|}.
\end{eqnarray*}
Hence,
\begin{eqnarray*}
|\nabla_y A(y,z)|\leq c_1+c_2, \quad |\nabla_z A(y,z)|\leq
c_2, \quad \forall (y,z)\in D\times D\;\;{\rm and}\;\; y\neq z.
\end{eqnarray*}
Since
\begin{eqnarray}\label{e58}
(y,z)\cdot\nabla A(y,z) &=& c_1\frac{y\cdot
(y-x)}{|y-x|}+c_2|y-z|\notag\\
&=&c_1|y|\cos\theta+c_2|y-z|\geq c_0>0,
\end{eqnarray}
where $\theta$ denotes the angle between $y$ and $y-x$, noting the facts that
the origin belongs to $U$ and $U$ is convex, we have $(y,z)\cdot\nabla A(y,z)$
has a positive lower bound for $(y,z)\in D\times D$ and $y\neq z$. So
\begin{eqnarray}\label{e59}
0<c'_1\leq |\nabla A(y,z)|\leq c'_2<\infty,\quad\forall (y,z)\in D\times
D\;\;{\rm and}\;\; y\neq z.
\end{eqnarray}

Our aim is to express $g_{\delta}(x,\omega)$ as a one-dimensional Fourier
transform and get rid of the variable $\omega$. To this end, we define the
following surface
\begin{eqnarray*}
\Gamma'_t:=\{(y,z)\in D\times D| A(y,z)=t\},\quad t>0.
\end{eqnarray*} 
It is easy to see that there exists smallest and largest values $T_0=T_0(x)$ and
$T_1=T_1(x)$ such that $\Gamma'_t$ is nonempty only for $t\in [T_0, T_1]$. Now
we fix a $\tilde{t}\in [T_0, T_1]$, then there exists $\eta=\eta(\tilde{t})$ and
an open cone $K=K(\tilde{t})\subset {\mathbb R^4}$ with center at the origin
such that for $t_0=\tilde{t}-\eta$ and $t_1=\tilde{t}+\eta$, we have 
\begin{eqnarray*}
D\times D\cap\{t_0<A(y,z)<t_1\}\subset K\cap\{t_0<A(y,z)<t_1\}:=\Gamma.
\end{eqnarray*}
Moreover, since $D$ has a positive distance to the origin we may also choose
$\eta$ and $K$ such that 
\begin{eqnarray}\label{e60}
|y|,|z|\geq c'_3>0 \quad\forall (y,z)\in \Gamma.
\end{eqnarray}
Denote $\Gamma_t=\Gamma\cap\{(y,z):A(y,z)=t\}$. We obtain
$\Gamma=\cup_{t_0\leq t\leq t_1}\Gamma_t$. By (\ref{e58}) and (\ref{e59}), we
deduce that there is a radial stretch $B_t$ yielding a bi-Lipschitz chart $B_t:
F\to \Gamma_t$ over a subdomain $F$ of the unit ball. The bi-Lip constant of
$B_t$ is uniform over $t_0<t<t_1$ and each $B_t$ is actually a local
diffeomorphism apart from $y=z$. By (\ref{e58}) and (\ref{e59}), we may write
$B_t$ in the following form
\begin{eqnarray*}\label{e61}
B_t(w_1,w_2)=\sigma(t,w_1,w_2)(w_1,w_2),
\end{eqnarray*}
where the dependence $(w_1,w_2)\to \sigma(t,w_1,w_2)$ is Lipschitz with respect
to $t$ with a uniform Lipschitz constant with respect to $w_1,w_2$.

Let $h$ be a integrable Borel-function on $\Gamma$, note that
$\Gamma=\cup_{t_0\leq t\leq t_1}\Gamma_t$, we get
\begin{eqnarray}\label{e62}
\int_{\Gamma}h(y,z)dydz=\int_{t_0}^{t_1}\int_{\Gamma_t}h(y,z)\frac{1}{|\nabla
A(y,z)|}d {\mathcal H^3}(y,z)dt,
\end{eqnarray}
where the inner integral is with respect to the three-dimensional Hausdorff
measure on $\Gamma_t$. By a change of variables, we have 
\begin{eqnarray}\label{e63}
\int_{\Gamma_t}h(y,z)d{\mathcal H^3}(y,z)=\int_F
h(B_t(w_1,w_2))E_t(w_1,w_2)d{\mathcal H^3}(w_1,w_2).
\end{eqnarray}
By (\ref{e58}) and (\ref{e59}), the Jacobian $E_t$ in \eqref{e63} satisfies
\begin{eqnarray*}
0<c'_4\leq E_t(w_1,w_2):=\frac{|B_t(w_1,w_2)|^3|\nabla
A(B_t(w_1,w_2))|}{|(w_1,w_2)\cdot \nabla A(B_t(w_1,w_2))|}\leq c'_5<\infty.
\end{eqnarray*} 
Since $B_t(w_1, w_2)$ is Lipschitz with respect to $t$, for our later purpose,
we claim that the dependence $t\to E_t(w_1,w_2)$ is uniformly Lipschitz with
respect to $t$. Using (\ref{e62}), we have
\begin{eqnarray*}
g_{\delta}(x,\omega)&=&\int_D\int_De^{{\rm
i}\omega(c_1|x-y|+c_2|y-z|)}\frac{(x_1-y_1)^{p_1}(x_2-y_2)^{p_2}(y_1-z_1)^{p_3}
(y_2-z_2)^{p_4}}{|x-y|^{l_1}|y-z|^{l_2}}\\
&&\quad \times q(y)\tilde{f}_{\delta}(z)dydz\\
&=&\int_{\Gamma}e^{{\rm
i}\omega(c_1|x-y|+c_2|y-z|)}\frac{(x_1-y_1)^{p_1}(x_2-y_2)^{p_2}(y_1-z_1)^{p_3}
(y_2-z_2)^{p_4}}{|x-y|^{l_1}|y-z|^{l_2}}\\
&&\quad \times q(y)\tilde{f}_{\delta}(z)dydz\\
&=&\int_{t_0}^{t_1}e^{{\rm i}\omega
t}S_{\delta}(t)dt=[\mathcal{F}^{-1}S_{\delta}](-\omega),
\end{eqnarray*}
where $S_{\delta}$ is given by
\begin{eqnarray*}
S_{\delta}(t)&=&\int_{\Gamma_t}\frac{(x_1-y_1)^{p_1}(x_2-y_2)^{p_2}(y_1-z_1)^{
p_3}(y_2-z_2)^{p_4}}{|x-y|^{l_1}|y-z|^{l_2}}\\
&&\quad\times \frac{1}{|\nabla
A(y,z)|}q(y)\tilde{f}_{\delta}(z)d {\mathcal H^3}(y,z).
\end{eqnarray*}
Since $\Gamma_t$ is only nonempty for $t\in[T_0, T_1]$, $S_{\delta}(t)$ is
compactly supported inside $[T_0, T_1]$. For fixed $x\in U$, let $L(x,y)$ be a
smooth cutoff of the
function $\frac{(x_1-y_1)^{p_1}(x_2-y_2)^{p_2}}{|x-y|^{l_1}}$ that vanishes
outside $D$, hence, $L(x,\cdot)\in C_0^{\infty}({\mathbb R^2})$. Thus, we can
rewrite $S_{\delta}(t)$ as 
\begin{eqnarray}\label{e64}
S_{\delta}(t)=\int_{\Gamma_t}\frac{(y_1-z_1)^{p_3}(y_2-z_2)^{p_4}}{|y-z|^{l_2}}
\frac{L(x,y)}{|\nabla A(y,z)|}q(y)\tilde{f}_{\delta}(z)d{\mathcal H^3}(y,z).
\end{eqnarray}
Recall that our aim is to prove
$\sup\limits_{\delta\in(0,1)}\int_1^{\infty}{\mathbb
E}|g_{\delta}(x,\omega)|^2d\omega<\infty.$ It is sufficient to show that for
each $\tilde{t}\in [T_0, T_1]$, there exists a finite constant
$M=M(\tilde{t})<\infty$ such that 
\begin{eqnarray}\label{e65}
{\mathbb E}|S_{\delta}(t)|^2\leq M,\quad\forall\delta\in (0,1)\;\;{\rm
and}\;\;t\in [t_0(\tilde{t}), t_1(\tilde{t})].
\end{eqnarray}
This can be seen by the following facts: by compactness, we can choose finitely
many $\tilde{t}\in [T_0, T_1]$ such that the union set of $[t_0(\tilde{t}),
t_1(\tilde{t})]$ for these $\tilde{t}$ can cover $[T_0, T_1]$.  Hence, for any
$t\in [T_0, T_1]$, we have ${\mathbb E}|S_{\delta}(t)|^2\leq M'$. The Parseval
formula yields 
\begin{eqnarray*}
\sup_{\delta\in(0,1)}\int_1^{\infty}{\mathbb
E}|g_{\delta}(x,\omega)|^2d\omega\lesssim
\sup_{\delta\in(0,1)}\int_{T_0}^{T_1}{\mathbb E}|S_{\delta}(t)|^2dt\leq
M'(T_1-T_0)<\infty.
\end{eqnarray*}

It remains to show (\ref{e65}). By (\ref{e64}), we have
\begin{eqnarray*}
{\mathbb
E}|S_{\delta}(t)|^2&=&\int_{\Gamma_t\times\Gamma_t}\frac{(y_1-z_1)^{p_3}
(y_2-z_2)^{p_4}}{|y-z|^{l_2}}\frac{(y'_1-z'_1)^{p_3}(y'_2-z'_2)^{p_4}}{|y'-z'|^{
l_2}}\\
&&\quad \times\frac{L(x,y)}{|\nabla A(y,z)|}\frac{L(x,y')}{|\nabla A(y',z')|}
q(y)q(y'){\mathbb E}(\tilde{f}_{\delta}(z)\tilde{f}_{\delta}(z'))d{\mathcal
H^3}(y,z)d{\mathcal H^3}(y',z').
\end{eqnarray*}
Noting that ${\mathbb
E}(\tilde{f}_{\delta}(z)\tilde{f}_{\delta}(z'))=C_{\delta}(z,z')$ and
$C_{\delta}=M_{\delta}C_{\tilde{f}}M_{\delta}$, we obtain from Lemma
\ref{lemma4} that for any given $\beta>0$, there is a finite constant
$C'_{\beta}$ such that $C_{\delta}(z,z')\leq C'_{\beta}|z-z'|^{-\beta}$ for any
$\delta\in (0,1)$ and $(z,z')\in D\times D$. Since  $q\in C_0^1(\overline{D})$,
an application of H\"{o}lder's inequality arrives  
\begin{align*}
\sup_{\delta\in (0,1)}{\mathbb E}|S_{\delta}(t)|^2&\lesssim
\int_{\Gamma_t\times\Gamma_t}|z-z'|^{-\beta}(|y-z||y'-z'|)^{-(l_2-p_3-p_4)}d{
\mathcal H^3}(y,z)d{\mathcal H^3}(y',z')\\
&\lesssim\bigg[\int_{\Gamma_t\times\Gamma_t}|z-z'|^{-2\beta}d{\mathcal
H^3}(y,z)d{\mathcal
H^3}(y',z')\bigg]^{\frac{1}{2}}\\
&\quad \times\bigg[\int_{\Gamma_t}|y-z|^{-1}d{\mathcal
H^3}(y,z)\int_{\Gamma_t}|y'-z'|^{-1}d{\mathcal H^3}(y',z')\bigg]^{\frac{1}{2}},
\end{align*}
where we use the fact $l_2-p_3-p_4=\frac{1}{2}$ for
$(l_1,l_2,p_1,p_2,p_3,p_4)\in S$. To show the integral in the right hand side of
the above inequality is bounded, we need the following result
\cite[Lemma 6]{LPS}).

\begin{lemma}\label{lemma9}
Given $\gamma\in(0,2)$ there is a finite constant $c$ such that for every $t\in
[t_0,t_1]$ we have
\begin{eqnarray*}
\int_{\Gamma_t}|y-z|^{-\gamma}d{\mathcal H^3}(y,z)\leq c, \quad 
\int_{\Gamma_t\times\Gamma_t}|\tilde{y}-\tilde{z}|^{-\gamma}d{\mathcal
H^3}(y,z)d{\mathcal H^3}(y',z')\leq c
\end{eqnarray*}
for $(\tilde{y},\tilde{z})=(y,y'), (y,z'),(z,y'),(z,z')$.
\end{lemma}

Choosing $\beta=\frac{1}{2}$ and applying Lemma \ref{lemma9} give
(\ref{e65}). So Theorem \ref{theorem5} holds for the case $m=2$.

Case 2. $m\in (2,5/2)$. By Lemma \ref{lemma3}, we know that in this case the
realizations of $\tilde{f}$ are H\"{o}lder continuous with probability one. So
it is not necessary to introduce the mollification, we define 
\begin{eqnarray*}
S(t)=\int_{\Gamma_t}\frac{(y_1-z_1)^{p_3}(y_2-z_2)^{p_4}}{|y-z|^{l_2}}\frac{L(x,
y)}{|\nabla A(y,z)|}q(y)\tilde{f}(z)d {\mathcal H^3}(y,z).
\end{eqnarray*}
In order to prove (\ref{e56}), i.e.,
$\int_1^{\infty}\omega^{m-2}|g(x,\omega)|^2d\omega<\infty$, note that
$g(x,\omega)=[{\mathcal F}^{-1}S](-\omega)$, it suffices to prove that $S(t)\in
H^{\frac{m-2}{2}}_{\rm homog}(\mathbb R)$ which denotes the homogeneous Sobolev
space. By compactness, it is enough to show that $S(t)\in H^{\frac{m-2}{2}}_{\rm
homog}(t_0(\tilde{t}), t_1(\tilde{t}))$ for each $\tilde{t}\in [T_0, T_1]$.
According to the Besov characterization of the homogeneous Sobolev space, it is
sufficient to show
\begin{eqnarray}\label{e66}
{\mathbb
E}\int_{t_0}^{t_1}\int_{t_0}^{t_1}\frac{|S(t)-S(t')|^2}{|t-t'|^{m-1}}
dtdt'<\infty.
\end{eqnarray}
The Fubini theorem shows that (\ref{e66}) holds as long as for some positive
constant $M$ that 
\begin{eqnarray}\label{e67}
{\mathbb E}|S(t)-S(t')|^2\leq M|t-t'|^{\frac{m-1}{2}},\quad\forall t,t'\in
[t_0(\tilde{t}), t_1(\tilde{t})]. 
\end{eqnarray}
We can rewrite $S(t)$ by 
\begin{eqnarray}\label{e68}
S(t)=\int_{\Gamma_t}N(y,z)L(x,y)\frac{1}{|\nabla A(y,z)|}q(y)\tilde{f}(z)d
{\mathcal H^3}(y,z).
\end{eqnarray}
Recall that the bi-Lipschitz chart $B_t:F\to \Gamma_t$ is given by
\begin{eqnarray*}
B_t(w_1,w_2)=\sigma(t,w_1,w_2)(w_1,w_2):=(y_t(w_1,w_2),z_t(w_1,w_2)).
\end{eqnarray*}
Denote
\begin{eqnarray*}
N_t(y,z) = \frac{(y_1-z_1)^{p_3}(y_2-z_2)^{p_4}}{|y-z|^{l_2}}.
\end{eqnarray*}
By (\ref{e68}), we can rewrite $S(t)$ by 
\begin{eqnarray*}\label{e69}
S(t)=\int_{F}N_t(y_t,z_t)T_t(w_1,w_2)q(y_t)\tilde{f}(z_t)d {\mathcal
H^3}(w_1,w_2),
\end{eqnarray*}
where the function 
\begin{eqnarray*}\label{e70}
T_t(w_1,w_2)=E_t(w_1,w_2)\frac{L(x,y_t)}{|\nabla A(y_t,z_t)|}
\end{eqnarray*}
is uniformly bounded and Lipschitz continuous with respect to $t$. Since 
\begin{eqnarray*}
S(t)-S(t')&=&S_1(t)-S_1(t')+\int_FN_{t'}(y_{t'},z_{t'})[T_t(w_1,w_2)-T_{t'}
(w_1, w_2)]d {\mathcal H^3}(w_1,w_2),
\end{eqnarray*}
where 
\begin{eqnarray*}
S_1(t)=\int_{F}N_t(y_t,z_t)T(w_1,w_2)q(y_t)\tilde{f}(z_t)d {\mathcal
H^3}(w_1,w_2),\quad T(w_1,w_2)=T_t(w_1,w_2),
\end{eqnarray*}
we have 
\begin{align*}
\|S(t)-S(t')\|_{L^2(\Omega)}\lesssim & \|S_1(t)-S_1(t')\|_{L^2(\Omega)}+\\
&\quad |t-t'|\int_F|q(y_{t'})|\|\tilde{f}(z_{t'})\|_{L^2(\Omega)}|N_{t'}(y_{t'},
z_{t'})|d {\mathcal H^3}(w_1,w_2)\\
\lesssim &\|S_1(t)-S_1(t')\|_{L^2(\Omega)}+|t-t'|.
\end{align*}
Since $|t-t'|=|t-t'|^{\frac{m-1}{2}}|t-t'|^{\frac{3-m}{2}}\lesssim
|t-t'|^{\frac{m-1}{2}}$, it suffices to estimate
$\|S_1(t)-S_1(t')\|_{L^2(\Omega)}$. Similarly, we have
\begin{eqnarray*}
S_1(t)-S_1(t')=S_2(t)-S_2(t')+\int_F[N_{t}(y_{t},z_{t})-N_{t'}(y_{t'},z_{t'})]
T(w_1,w_2)d {\mathcal H^3}(w_1,w_2),
\end{eqnarray*}
where
\begin{eqnarray*}
S_2(t)=\int_{F}N(w_1,w_2)T(w_1,w_2)q(y_t)\tilde{f}(z_t)d {\mathcal
H^3}(w_1,w_2),\quad N(w_1,w_2)=N_t(w_1,w_2).
\end{eqnarray*}
Note that
\begin{align*}
&|N_{t}(y_{t},z_{t})-N_{t'}(y_{t'},z_{t'})|\\
=&\left|\frac{(y_1(t)-z_1(t))^{p_3}(y_2(t)-z_2(t))^{p_4}}{|y(t)-z(t)|^{l_2}}
-\frac{(y_1(t')-z_1(t'))^{p_3}(y_2(t')-z_2(t'))^{p_4}}{|y(t')-z(t')|^{l_2}}
\right|\\
=&\left|\frac{\sigma_t^{p_3}(w_1^{(1)}-w_2^{(1)})^{p_3}\sigma_t^{p_4}(w_1^{(2)}
-w_2^{(2)})^{p_4}}{\sigma_t^{l_2}|w_1-w_2|^{l_2}}-\frac{\sigma_{t'}^{p_3}(w_1^{
(1)}-w_2^{(1)})^{p_3}\sigma_{t'}^{p_4}(w_1^{(2)}-w_2^{(2)})^{p_4}}{\sigma_{t'}^{
l_2}|w_1-w_2|^{l_2}}\right|\\
\leq&
|\sigma_t^{-\frac{1}{2}}-\sigma_{t'}^{-\frac{1}{2}}||w_1-w_2|^{-\frac{1}{2}}
\lesssim |t-t'||w_1-w_2|^{-\frac{1}{2}}.
\end{align*}
Hence 
\begin{eqnarray*}\label{e72}
\|S_1(t)-S_1(t')\|_{L^2(\Omega)}\lesssim
\|S_2(t)-S_2(t')\|_{L^2(\Omega)}+|t-t'|.
\end{eqnarray*}
Now we estimate $\|S_2(t)-S_2(t')\|_{L^2(\Omega)}$ which can be rewritten in a
double integral as
\begin{eqnarray*}
&& \|S_2(t)-S_2(t')
 |_{L^2(\Omega)}
 ={\mathbb E}\int_F[q(y_t)\tilde{f}(z_t)-q(y_{t'})\tilde{f}(z_{t'})]R(w_1,w_2)d
{\mathcal H^3}(w_1,w_2)\\
& &\quad\times
\int_F[q(s_t)\tilde{f}(u_t)-q(s_{t'})\tilde{f}(u_{t'})]R(v_1,v_2)d
{\mathcal H^3}(v_1,v_2)\\
& =&\int_{F\times F}G(w_1,w_2,v_1,v_2)R(w_1,w_2)R(v_1,v_2)d {\mathcal
H^3}(w_1,w_2)d {\mathcal H^3}(v_1,v_2),
\end{eqnarray*}
where $(y_t, z_t)=\sigma_t(w_1, w_2)$, $(y_{t'}, z_{t'})=\sigma_{t'}(w_1, w_2)$,
$(s_t, u_t)=\sigma_t(v_1, v_2)$, $(s_{t'}, u_{t'})=\sigma_{t'}(v_1, v_2)$, 
\begin{align*}
 R(w_1, w_2)=N(w_1, w_2)T(w_1, w_2),\quad R(v_1, v_2)=N(v_1, v_2)T(v_1, v_2),
\end{align*}
and 
\begin{eqnarray*}
&&G(w_1,w_2,v_1,v_2)={\mathbb
E}[q(y_t)\tilde{f}(z_t)-q(y_{t'})\tilde{f}(z_{t'})][q(s_t)\tilde{f}(u_t)-q(s_{t'
})\tilde{f}(u_{t'})]\\
&=&q(y_t)q(s_t)C_{\tilde{f}}(z_t,u_t)-q(y_t)q(s_{t'})C_{\tilde{f}}(z_t,u_{t'}
)\\
&&\quad
-q(y_{t'})q(s_t)C_{\tilde{f}}(z_{t'},u_t)+q(y_{t'})q(s_{t'})C_{\tilde{f}}
(z_ { t' },u_{t'})\\
&=&q(y_t)q(s_t)[C_{\tilde{f}}(z_t,u_t)-C_{\tilde{f}}(z_t,u_{t'})]+q(y_t)[
q(s_t)-q(s_{t'})]C_{\tilde{f}}(z_t,u_{t'})\\
&&\quad+q(y_{t'})q(s_{t'})[C_{\tilde{f}}(z_{t'},u_{t'})-C_{\tilde{f}}(z_{t'},
u_t)
]+q(y_{t'})[q(s_{t'})-q(s_t)]C_{\tilde{f}}(z_{t'},u_t).
\end{eqnarray*}
Recall that the covariance function has the form 
\begin{eqnarray*}
C_{\tilde{f}}(y,z)=c_0(y,z)|y-z|^{m-2}+r_1(y,z),
\end{eqnarray*}
where $c_0\in C_0^{\infty}(D\times D)$ and $r_1\in C_0^{\alpha}(D\times D)$ for
any $\alpha<1$. Combining the fact $q\in C_0^{1}(\overline{D})$ yields
immediately that
\begin{eqnarray}\label{e73}
|G(w_1,w_2,v_1,v_2)|\lesssim |t-t'|^{m-2}.
\end{eqnarray}  
Denote $d=|z_t-u_t|=|\sigma_t(w_2-v_2)|$ and
$\delta=|u_t-u_{t'}|=|(\sigma_t-\sigma_{t'})v_2|$, if $\frac{\delta}{d}<1$, we
have
\begin{eqnarray*}
&&\left||z_t-u_t|^{m-2}-|z_t-u_{t'}|^{m-2}\right|\leq\left|(d+\delta)^{m-2}-d^{
m-2}\right|
=d^{m-2}\left|(1+\frac{\delta}{d})^{m-2}-1\right|\\
&&\leq d^{m-2}(m-2)\frac{\delta}{d}
=(m-2)d^{m-3}\delta\lesssim \delta^{\frac{m-1}{2}}\lesssim
|t-t'|^{\frac{m-1}{2}}.
\end{eqnarray*}
Hence, if $|t-t'|<c|w_2-v_2|$ for some small enough $c>0$, we have 
\begin{eqnarray*}
\left||z_t-u_t|^{m-2}-|z_t-u_{t'}|^{m-2}\right|\lesssim |t-t'|^{\frac{m-1}{2}}.
\end{eqnarray*}
Similarly, we have that 
\begin{eqnarray*}
\left||z_{t'}-u_{t'}|^{m-2}-|z_{t'}-u_t|^{m-2}\right|\lesssim
|t-t'|^{\frac{m-1}{2}}
\end{eqnarray*}
holds if $|t-t'|<c|w_2-v_2|$ for some small enough $c>0$. Thus, if we define a
set
\begin{eqnarray*}
P:=\{(w_1,w_2,v_1,v_2)\in F\times F: |w_2-v_2|\leq C|t-t'|\quad\text{for some
large enough } C>0 \},
\end{eqnarray*}
then we have
\begin{eqnarray}\label{e74}
|G(w_1,w_2,v_1,v_2)|\lesssim |t-t'|^{\frac{m-1}{2}} \quad{\rm for}~
(w_1,w_2,v_1,v_2)\in F\times F\setminus P.
\end{eqnarray}
Dividing integration on $F\times F$ over the sets $P\cap F\times F$ and
$(F\times F)\setminus P$, we obtain
\begin{align*}
 &\|S_2(t)-S_2(t')\|_{L^2(\Omega)}\\
 &=\int_{F\times F\cap P}G(w_1,w_2,v_1,v_2)R(w_1,w_2)R(v_1,v_2)d {\mathcal
H^3}(w_1,w_2)d {\mathcal H^3}(v_1,v_2)\\
&\quad +\int_{(F\times F)\setminus P}G(w_1,w_2,v_1,v_2)R(w_1,w_2)R(v_1,v_2)d
{\mathcal H^3}(w_1,w_2)d {\mathcal H^3}(v_1,v_2)\\
&:=I_1+I_2.
\end{align*}
Observe that $|R(w_1,w_2)|\lesssim |w_1-w_2|^{-\frac{1}{2}}$ and
$|R(v_1,v_2)|\lesssim |v_1-v_2|^{-\frac{1}{2}}$, using (\ref{e73}), H\"{o}lder
inequality along with Lemma \ref{lemma9}, we have 
\begin{eqnarray*}
I_1&\lesssim & |t-t'|^{m-2}\int_{F\times F\cap
P}|w_1-w_2|^{-\frac{1}{2}}|v_1-v_2|^{-\frac{1}{2}}d {\mathcal H^3}(w_1,w_2)d
{\mathcal H^3}(v_1,v_2)\\
&\lesssim & |t-t'|^{m-2}\int_{F\times F\cap
P}|w_2-v_2|^{\frac{1}{2}}|w_2-v_2|^{-\frac{1}{2}}|w_1-w_2|^{-\frac{1}{2}}\\
&&\quad \times |v_1-v_2|^{-\frac{1}{2}}d {\mathcal H^3}(w_1,w_2)d {\mathcal
H^3}(v_1,v_2)\\
&\lesssim &|t-t'|^{m-\frac{3}{2}}\int_{F\times F\cap
P}|w_2-v_2|^{-\frac{1}{2}}|w_1-w_2|^{-\frac{1}{2}}|v_1-v_2|^{-\frac{1}{2}}d
{\mathcal H^3}(w_1,w_2)d {\mathcal H^3}(v_1,v_2)\\
&\lesssim &|t-t'|^{\frac{m-1}{2}+\frac{m-2}{2}}\left(\int_{F\times F\cap
P}|w_2-v_2|^{-\frac{3}{2}}d {\mathcal H^3}(w_1,w_2)d {\mathcal
H^3}(v_1,v_2)\right)^{\frac{1}{3}}\times\\
&&\quad \left(\int_{F\times F\cap P}|w_1-w_2|^{-\frac{3}{2}}d {\mathcal
H^3}(w_1,w_2)\right)^{\frac{1}{3}}\left(\int_{F\times F\cap
P}|v_1-v_2|^{-\frac{3}{2}}d {\mathcal H^3}(v_1,v_2)\right)^{\frac{1}{3}}\\
&\lesssim & |t-t'|^{\frac{m-1}{2}}.
\end{eqnarray*}
For $I_2$,  we have from (\ref{e74}) that 
\begin{eqnarray*}
I_2&\lesssim &|t-t'|^{\frac{m-1}{2}}\int_{(F\times F)\setminus
P}|w_1-w_2|^{-\frac{1}{2}}|v_1-v_2|^{-\frac{1}{2}}d {\mathcal H^3}(w_1,w_2)d
{\mathcal H^3}(v_1,v_2)\\
&\lesssim & |t-t'|^{\frac{m-1}{2}}\left(\int_{(F\times F)\setminus
P}|w_1-w_2|^{-1}d {\mathcal
H^3}(w_1,w_2)\right)^{\frac{1}{2}}\\
&&\quad\times\left(\int_{(F\times F)\setminus
P}|v_1-v_2|^{-1}d {\mathcal H^3}(v_1,v_2)\right)^{\frac{1}{2}}\\
&\lesssim & |t-t'|^{\frac{m-1}{2}},
\end{eqnarray*} 
where we use the H\"{o}lder inequality along with Lemma \ref{lemma9}. Hence, we
arrive  
\begin{eqnarray*}
\|S_2(t)-S_2(t')\|_{L^2(\Omega)}\lesssim |t-t'|^{\frac{m-1}{2}},
\end{eqnarray*}
which shows that (\ref{e67}) holds true. By the previous argument we have that
(\ref{e56}) holds for this case. The proof is completed.
\end{proof}

With the convergence of the Born approximation, using Theorem \ref{theorem4}
along with Theorem \ref{theorem5},  we are ready to show the proof of Theorem
\ref{theorem1}. 

\begin{proof}
Recall the convergence of the Born approximation
\begin{eqnarray*}
\bm{u}(x,\omega)=\bm{u}_0(x,\omega)+\bm{u}_1(x,\omega)+\bm{b}(x,\omega),
\end{eqnarray*}
where $\bm{b}(x,\omega)=\sum_{n=2}^{\infty}\bm{u}_n(x,\omega)$. It
follows from (\ref{d10}) that 
\begin{eqnarray*}
\|\bm{b}(x,\omega)\|_{L^{\infty}(U)^2}\lesssim \omega^{-2+\varepsilon'},
\end{eqnarray*} 
for some small enough $\varepsilon'>0$. So
\begin{eqnarray}\label{e75}
\frac{1}{Q-1}\int_1^Q\omega^{m+1}|\bm{b}(x,
\omega)|^2{\rm d}\omega\lesssim
\frac{1}{Q-1}\int_1^Q\omega^{m-3+2\varepsilon'}d\omega\to 0,
\end{eqnarray}
as $Q\to\infty$, where we use the fact $m\in (2,5/2)$.
Recalling Theorem \ref{theorem4} and Theorem \ref{theorem5}, we have 
\begin{align}\label{e76}
&\lim_{Q\rightarrow\infty}\frac{1}{Q-1}\int_1^Q\omega^{m+1}|\bm{u}_0(x,
\omega)|^2 d\omega=a\int_{\mathbb
R^2}\frac{1}{|x-y|}\phi(y)dy,\\\label{e77}
&\lim_{Q\rightarrow\infty}\frac{1}{Q-1}\int_1^Q\omega^{m+1}|\bm{u}_1(x,
\omega)|^2 d\omega=0.
\end{align}
hold almost surely, where $a$ is a constant given in Theorem \ref{theorem1}.
Since
\begin{align*}
|\bm{u}(x,\omega)|^2=&|\bm{u}_0(x,\omega)|^2+|\bm{u}_1(x,\omega)|^2+|\bm{b}(x,
\omega)|^2\\
&+2\Re[\bm{u}_0(x,\omega)\overline{\bm{u}_1(x,\omega)}]+2\Re[\bm{u}_0(x,
\omega)\overline{\bm{b}(x,\omega)}]+2\Re[\bm{u}_1(x,\omega)\overline{\bm{b}(x,
\omega)}],
\end{align*}
along with (\ref{e75})--(\ref{e77}) and the Cauchy-Schwartz inequality, it is
to easy to verify that 
\begin{eqnarray*}
\lim_{Q\to\infty}\frac{1}{Q-1}\int_1^Q\omega^{m+1}|\bm{u}(x,
\omega)|^2 d\omega=a\int_{\mathbb
R^2}\frac{1}{|x-y|}\phi(y)dy.
\end{eqnarray*}
By Lemma 3.8 in \cite{LHL}, we know that the integral $\int_{\mathbb
R^2}\frac{1}{|x-y|}\phi(y)dy$ for all $x\in U$ can uniquely determines the
function $\phi$. The proof is completed. 
\end{proof}

\section{Conclusion}

We have studied the inverse random source scattering problem for the
two-dimensional elastic wave equation with an inhomogeneous, anisotropic mass
density. The source is modeled as a generalized Gaussian random function and its
covariance operator is described as a classical pseudo-differential operator.
Both the direct and the inverse problems are considered. The direct problem is
equivalently formulated as a Lippmann--Schwinger integral equation which is
shown to have a unique solution. Combining the Born approximation and microlocal
analysis, we deduce a relationship between the principle symbol of the
covariance operator for the random source and the amplitude of the displacement
generated from a single realization of the random source. Based on this
connection, we obtain the uniqueness for the reconstruction of the principle
symbol of the random source. In this paper, the mass density or the linear load
is considered to be a smooth deterministic matrix. An ongoing project is to
study the direct and inverse scattering problems when both the source and the
mass density or the linear load are random. Another challenging problem is to
study the random source scattering problem for three-dimensional elastic wave
equation. We hope to be able to report the progress elsewhere in the future.

\end{document}